\documentclass{amsart}

\usepackage{longtable}

\usepackage{amsfonts}
\usepackage{amsmath}
\usepackage{amssymb}
\usepackage[mathscr]{eucal}
\usepackage{color}
\usepackage{url}
\usepackage{cases}
\usepackage{bm}
\usepackage{comment}
\usepackage{multirow, tabularx}

\usepackage{caption}

\usepackage{amsthm}

\theoremstyle{plain}
 \newtheorem{theorem}{Theorem}[section]
 
 \newtheorem{result}{Result}[section] 
 \newtheorem*{theorem*}{Theorem}
 \newtheorem{proposition}[theorem]{Proposition}
 \newtheorem{lemma}[theorem]{Lemma}
 \newtheorem{corollary}[theorem]{Corollary}

\theoremstyle{definition}
 
\theoremstyle{remark}
 \newtheorem{remark}[theorem]{Remark}
  \newtheorem{example}[theorem]{Example}
\numberwithin{equation}{section}

\def\alg#1{\textcolor{red}{\bf #1}}

\newcommand{\OA}{\ensuremath{\operatorname{OA}}}
\newcommand{\sV}[2]{{
	\setlength{\arraycolsep}{2pt}
	\renewcommand{\arraystretch}{0.8}
	\left[\begin{array}{ccc} #1 \\ #2 \end{array}\right]
}}

\DeclareUnicodeCharacter{2212}{-}
\makeatletter
\def\bbordermatrix#1{\begingroup \m@th
  \@tempdima 4.75\p@
  \setbox\z@\vbox{%
    \def\cr{\crcr\noalign{\kern2\p@\global\let\cr\endline}}%
    \ialign{$##$\hfil\kern2\p@\kern\@tempdima&\thinspace\hfil$##$\hfil
      &&\quad\hfil$##$\hfil\crcr
      \omit\strut\hfil\crcr\noalign{\kern-\baselineskip}%
      #1\crcr\omit\strut\cr}}%
  \setbox\tw@\vbox{\unvcopy\z@\global\setbox\@ne\lastbox}%
  \setbox\tw@\hbox{\unhbox\@ne\unskip\global\setbox\@ne\lastbox}%
  \setbox\tw@\hbox{$\kern\wd\@ne\kern-\@tempdima\left[\kern-\wd\@ne
    \global\setbox\@ne\vbox{\box\@ne\kern2\p@}%
    \vcenter{\kern-\ht\@ne\unvbox\z@\kern-\baselineskip}\,\right]$}%
  \null\;\vbox{\kern\ht\@ne\box\tw@}\endgroup}
\makeatother

\title{Extremal orthogonal arrays}
\author{Alexander L. Gavrilyuk}
\address{Department of Mathematical Sciences, University of Memphis, Tennessee, U.S.A.}
\email{a.gavrilyuk@memphis.edu}
\author{Sho Suda}
\address{Department of Mathematics, National Defense Academy of Japan, Kanagawa, Japan}
\email{ssuda@nda.ac.jp}

\begin{document}

\begin{abstract}
It is known that a Delsarte 
$t$-design in a $Q$-polynomial 
association scheme has degree at least 
$\left \lceil{\frac{t}{2}}\right \rceil
$. %\alg{(?)} This is true. 
Following Ionin and Shrikhande who studied combinatorial $(2s-1)$-designs 
(i.e., Delsarte designs in Johnson association schemes) 
having exactly $s$ block intersection numbers, we call a Delsarte 
$(2s-1)$-design 
with degree $s$ \emph{extremal} 
and study 
extremal orthogonal arrays, which 
are Delsarte designs in Hamming 
association schemes.

It was shown by Delsarte that a $t$-design with degree $s$ and $t\geq 2s-2$ in a Hamming association scheme 
induces an $s$-class association scheme. We prove 
that an extremal orthogonal array 
gives rise to a fission scheme 
of the latter one, which has 
$2s-1$ or $2s$ classes. 
As a corollary, %a result by Calderbank and Goethals on $(2s-2)$-designs with degree $s$ is strengthened, and 
a new necessary condition for the existence of tight orthogonal arrays 
of strength $3$ is obtained.

Furthermore, as a counterpart to 
a result of Ionin and Shrikhande, 
we prove an %two-sided 
inequality for Hamming distances in extremal orthogonal arrays.
%Both sides of 
The inequality is 
tight as shown 
by examples related to the Golay codes.
%the %ternary 
%Golay codes of length 11 and 12.
\end{abstract}

\maketitle

%%%%%%%%%%%%%%%%%%%%%%%%%%%%%%%%%%%%%%%%%%%%%%%%%%%%%%%%%%%%%%%%%%%%%%%%

\section{Introduction}

The concept of an algebraic 
(or \textbf{Delsarte}) $t$-design in $Q$-polynomial association schemes 
was introduced in \cite{D}. 
It unifies the notions of combinatorial 
(block) designs, orthogonal arrays, 
spherical designs. 
For an arbitrary $Q$-polynomial 
association scheme having a $t$-design 
with degree $s$ and $t\geq 2s-2$, Delsarte 
showed that such a design induces another 
$Q$-polynomial association scheme 
with $s$ classes, which will be referred to as a \textbf{Delsarte scheme} of the design.

One of the challenging problems in 
Algebraic Combinatorics is 
to classify tight designs in various settings, such as the Hamming association schemes $H(n,q)$, the Johnson association schemes $J(v,k)$, or spherical designs on the real unit sphere 
$S^{d-1}$. A design is \textbf{tight} if its size achieves a Rao type or Fisher type 
lower bound (see \cite{R,RW,DGS} 
for $H(n,q)$, $J(v,k)$, $S^{d-1}$ respectively).

Note that a tight $2s$-design necessarily has degree $s$. 
%\textcolor{red}{OR: Note that a $2s$-design with degree $s$ is 
%necessarily tight?}\textcolor{blue}{yes, this is also true. I prefer the former sentence, not red one.}
A classical necessary existence condition for the corresponding 
Delsarte scheme is given by 
the so-called Wilson type theorems
(\cite[Theorem~5.21]{D}), 
which, loosely speaking, state that 
the degree set (the $s$ intersection
numbers) of such a design must be 
(nonnegative integer) roots of a 
certain polynomial.
This condition restricts possible 
Hamming distances in tight 
orthogonal arrays, sizes of block intersections in tight combinatorial designs, 
and inner products of vectors in tight spherical designs, and it was 
a primary tool to show many 
classification results \cite{BD1,BD2,H}. 

It is also known \cite{D}
that a Delsarte $(2s-1)$-design has 
at least $s$ intersection numbers.
Combinatorial $(2s-1)$-designs 
having exactly $s$ block 
intersection numbers were studied 
by Ionin and Shrikhande \cite{IS}, 
who called 
them \emph{extremal}.
Here we adopt this terminology 
and study \textbf{extremal} orthogonal 
arrays, i.e., Delsarte $(2s-1)$-designs 
in Hamming association schemes 
%, that have 
with exactly $s$ Hamming distances.

%tight OA 
Recall that an \textbf{orthogonal array} %with parameters $(N,n,q,t)$ 
OA$(N,n,q,t)$ %for short) 
is an $N\times n$ matrix with entries from 
the alphabet $[q]:=\{1,2,\ldots,q\}$ arranged in such a way that in 
every $t$ columns all possible rows of $[q]^t$ occur equally often.
Since orthogonal arrays were introduced by Rao \cite{R}, they 
became one of the central topics in combinatorics \cite{HSS}. 
Given strength $t$, the alphabet size $q$, and the number of columns $n$, 
a fundamental problem is constructing orthogonal arrays with 
minimum possible number of rows~$N$. 
The lower bound on $N$ was given by Rao~\cite{R} as follows:
\begin{numcases}{N\geq}
    {\displaystyle \sum_{k=0}^e \binom{n}{k}(q-1)^k,\text{ if }t=2e,}%\text{\textcolor{red}{why e?}\textcolor{blue}{$\rightarrow$previously it is $t=2s$, but I want to use $s$ as degree.}}
    \label{eq:tight}\\
    {\displaystyle \sum_{k=0}^{e-1} \binom{n}{k}(q-1)^k+\binom{n-1}{e-1}(q-1)^{e},\text{ if }t=2e-1}.\label{eq:tighto}
\end{numcases}
An orthogonal array is said to be \textbf{tight} %(or \textbf{complete})
if it satisfies Eq.\ \eqref{eq:tight} or \eqref{eq:tighto} with equality. 
By taking the rows of the matrix as points, an $\OA(N,n,q,t)$ can be viewed as a Delsarte $t$-design 
in the Hamming association scheme $H(n,q)$. Furthermore, if it is tight 
%and $t=2s$, 
then the corresponding 
design has degree $e$.

Tight orthogonal arrays with 
$t=2$ and $q=2$ correspond 
to Hadamard matrices; thus, their classification seems hopeless. 
For $t=2$ and $q\geq 3$, 
tight $\OA(N,n,q,2)$ are only known 
with $n=q+1$ and $q$ a prime power 
\cite{B}.
However, for $t>2$, 
the existence and classification problem of tight OAs with even strength 
has received considerable attention; 
its current state is summarized as follows, see \cite{GSV}.

\begin{result}\label{theo:class2e} 
The following holds.
\begin{itemize} 
    %\item[$(1)$] For $t=2$ and $q\geq 3$, 
    %there exists a tight $\OA(N,n,q,2)$ with $n=q+1$ and $q$ a prime power.%\alg{(if and only if??)}
    \item[$(1)$] If $t=4$, then a tight $\OA(N,n,q,4)$ is one of the following:
    \begin{itemize}
        \item[$(i)$] the dual code of the binary repetition code of length $5$, i.e., $(N,n,q)=(16,5,2)$;
        \item[$(ii)$] the dual code of the ternary Golay code, i.e., 
        $(N,n,q)=(243,11,3)$.
    \end{itemize}
    \item[$(2)$] If $t=6$, then 
    $q=2$ and
    a tight $\OA(N,n,2,6)$ is one 
    of the following:
    \begin{itemize}
        \item[$(i)$] the dual code of the binary repetition code of length $7$, i.e., $(N,n)= (64,7)$;
        %$\OA(2^6,7,2,6)$ which is an extension of 
        %a trivial $\OA(2^6,6,2,6)$; 
        \item[$(ii)$] the dual code of the binary Golay code, 
        i.e., $(N,n)= (2048,23)$.
    \end{itemize}
    \item[$(3)$] There is no tight $\OA$ for all $t=2e\geq 6$ and $q\geq 3$. 
\end{itemize}
\end{result}

Note that \emph{contracting} an  
$\OA(N,n,q,t)$, i.e., taking the set of its rows having 
the same symbol in a fixed column, produces an %orthogonal 
%array with parameters 
$\OA(\frac{1}{q}N,n-1,q,t-1)$. 
In particular, a contraction of a tight OA with strength $t=2e+1$ 
is a tight OA with strength $t=2e$. Thus, in view of 
Result \ref{theo:class2e}, 
the classification problem 
of tight OAs %with odd degree 
makes sense in the following cases: $q\geq 2$ and $t=3$, 
or $q\in \{2,3\}$ and $t=5$,
or $q=2$ and $t\geq 7$. 
These cases were studied in  
\cite{MK94,N1986}, as summarized below. 
It is conjectured in \cite{MK94} 
that for all $t\geq 8$ a tight $\OA(N,n,2,t)$ exists 
if and only if $n=t+1$.

\begin{result}\label{odd-t}
The following holds. 
    \begin{itemize}
        \item[$(1)$] %(\cite{N1986}) 
        If $t=3$, then 
        a tight $\OA(N,n,q,3)$ satisfies one of the following cases:
    \begin{itemize}
        \item[(i)] $(N,n,q)=(2n,n,2)$ with $n\equiv 0\pmod{4}$; such an orthogonal array exists if and only if there exists a Hadamard matrix of order $n$;
        \item[(ii)] $(N,n,q)=(q^3,q+2,q)$ with $q$ even.
    \end{itemize}
    \item[$(2)$]  %(\cite{N1986}) 
    If $t=5$, then 
    a tight $\OA(N,n,q,5)$ is one 
    of the following:
    \begin{itemize}
        \item[$(i)$] the dual code of the binary repetition code of length $6$, i.e., $(N,n,q)= (32,6,2)$; 
        \item[$(ii)$] the dual code of the extended ternary Golay code, 
        i.e., $(N,n,q)= (729,12,3)$.
    \end{itemize}
    \item[$(3)$] %(\cite{MK94}) 
    If $t=7$, then $q=2$ and
    a tight $\OA(N,n,2,7)$ is one 
    of the following:
    \begin{itemize}
        \item[$(i)$] the dual code of the binary repetition code of length $8$, i.e., $(N,n)= (128,8)$;
        \item[$(ii)$] the dual code of the extended binary Golay code, i.e., $(N,n)= (4096,24)$.
    \end{itemize}
    \item[$(4)$] %(\cite{MK94}) 
    There is no 
    $\OA(N,n,2,t)$ with $8\leq t\leq 13$, $n\ne t+1$ 
    and $n\leq 10^9$.
    \end{itemize}
\end{result}

\begin{comment}
The case $t=5$ was settled 
by Noda \cite{N1986} as follows.

\begin{result}
    A tight $\OA(N,n,q,5)$ is one 
    of the following:
    \begin{itemize}
        \item[$(i)$] the dual code of the binary repetition code of length $6$, i.e., $(N,n,q)= (32,6,2)$; 
        \item[$(ii)$] the dual code of the extended ternary Golay code, 
        i.e., $(N,n,q)= (729,12,3)$.
    \end{itemize}
\end{result}

The case $q=2$ and $t\geq 7$ was considered in \cite{MK94}.
They conjectured that for all $t\geq 8$ a tight $\OA(N,n,2,t)$ exists 
if and only if $n=t+1$ and 
partially confirmed this as follows.

\begin{result}\label{theo:q=2}
The following holds.
\begin{itemize}
    \item[$(1)$] A tight $\OA(N,n,2,7)$ satisfies one of the following cases:
    \begin{itemize}
        \item[$(i)$] the dual code of the binary repetition code of length $8$, i.e., $(N,n)= (128,8)$;
        \item[$(ii)$] the dual code of the extended binary Golay code, i.e., $(N,n)= (4096,24)$.
    \end{itemize}
    \item[$(2)$] There is no 
    $\OA(N,n,2,t)$ with $8\leq t\leq 13$, $n\ne t+1$ 
    and $n\leq 10^9$.
\end{itemize}
\end{result}

For $t=3$, again 
Noda \cite{N1986} obtained the following.

\begin{result}\label{theo:class3}
    A tight $\OA(N,n,q,3)$ satisfies one of the following cases:
    \begin{itemize}
        \item[(i)] $(N,n,q)=(2n,n,2)$ with $n\equiv 0\pmod{4}$; such an orthogonal array exists if and only if there exists a Hadamard matrix of order $n$;
        \item[(ii)] $(N,n,q)=(q^3,q+2,q)$ with $q$ even.
    \end{itemize}
\end{result}
\end{comment}

The orthogonal arrays in Result \ref{odd-t} (1-ii) are 
known to exist if $q$ is a power of~2 \cite[Section~9]{BB}. 
Our first result in this paper 
refines this part of 
Noda's result \cite{N1986}
by showing (see Corollary \ref{cor:tight3}) 
that $q>2$ must be a multiple of four. 
To prove this, 
%(see Section \ref{sect:Definitions} for 
%all unexplained definitions).
we take a closer look at the Delsarte scheme of an extremal 
orthogonal array of strength 
$t=2s-1$. 
%Namely, assume that an $\OA(N,n,q,t)$ 
%is extremal, i.e., $t=2s-1$ and  
%the corresponding design in  
%$H(n,q)$ has degree $s$.
We show in Theorem \ref{thm:qant4} that this Delsarte scheme %of an extremal OA  %with degree $s$ 
%in a Hamming association scheme 
admits a \emph{fission scheme} 
with $2s-1$ or $2s$ classes 
(depending on whether the given OA is tight or no) and then examine its feasibility, using so-called triple intersection numbers. This extends the approach of how the classification of tight OAs with strength $4$
was completed in our joint work with Vidali \cite{GSV}. 
%Noda \cite{N1979} determined the possible parameters of such OAs, 
%while in \cite{GSV} the remaining putative parameters were ruled out by investigating so-called triple intersection numbers of the fission 
%(4-class) scheme of the Delsarte (2-class) scheme. 

Considering extremal combinatorial 
designs, Ionin and Shrikhande \cite{IS}
characterized, 
under certain assumptions, 
the Witt 5-$(24,8,1)$  
and 4-$(23,7,1)$ designs 
as the only extremal 
$(2s-1)$- and $2s$-designs,
respectively.
Furthermore, they proved the 
following inequality for 
block intersection numbers.

\begin{comment}
\begin{result}\label{theo:class2e_combindes} 
The following holds.
\begin{itemize} 
    \item[$(1)$] For $t=2$, there are infinitely many examples of symemtric designs \textcolor{red}{How to claim here? In OAs case, it seems only one parameter set for the examples, but there are too many for symmetric designs.}
    \item[$(2)$] If $t=4$, then a tight $4$-design is one of the following:
    \begin{itemize}
        \item[$(i)$] the Witt $2$-$(23,7,1)$ design;
        %\item[$(ii)$] the Witt %$2$-$(23,16,52)$ design.
    \end{itemize}
    \item[$(3)$] There is no tight $t$-design for all $t=2s\geq 6$. 
\end{itemize}
\end{result}
\end{comment}

\begin{result}\label{theo:ISineq} 
Let $\mathcal{B}\subset \binom{v}{k}$ be a $(2s-1)$-design with intersection numbers 
\mbox{$x_1<\cdots<x_s$}, that is, $\{|b\cap b'| \mid b,b'\in \mathcal{B}, b \neq b'\}=\{x_1,\ldots,x_s\}$.  
Then: 
\[
\frac{(s-1)(k-s)(k-s+1)}{v-2s+2}\leq \sum_{i=1}^s x_i-\frac{s(s-1)}{2} \leq \frac{s(k-s)(k-s+1)}{v-2s+1},
\]
with equality in the left if and only if $x_1=0$, and with equality in the right if and only if $\mathcal{B}$ is a tight $2s$-design. 
\end{result}

Our second result (Theorem \ref{thm:e2}) 
is a similar two-sided inequality 
for Hamming distances between rows 
of an extremal OA. Furthermore, the ternary 
Golay codes in $H(12,3)$ 
and $H(11,3)$ show that this inequality is tight (see Example \ref{extremalOA}).

\begin{comment}
This result allows us to strengthen a theorem by Calderbank and Goethals 
\cite[Theorem~1]{CG} (see Theorem~\ref{thm:CG}) under the assumption that the design is 
a $(2s-1)$-design, and we derive some further corollaries from it. As an application, 
we consider the existence problem of tight $3$-designs in the Hamming association schemes. 
\end{comment}

The paper is organized as follows. In Section~\ref{sect:Definitions}, we review some basic theory of association schemes and related concepts.  
In Section~\ref{sect:fission}, we consider the Delsarte scheme of 
an extremal OA of strength $2s-1$
and construct 
a fission $\mathcal{S}$ of this scheme, 
which has $2s-1$ or $2s$ classes, 
and determine 
its second eigenmatrix. 
In Section~\ref{sect:tight3}, 
we investigate the triple intersection numbers of the scheme $\mathcal{S}$ obtained from tight $3$-designs in $H(q+2,q)$ to show that $q$ must be a multiple of four if $q>2$. 
In Section~\ref{Section4}, 
we prove a Ionin-Shrikhande type 
inequality, determine some feasible parameters and known examples of 
extremal orthogonal arrays, and 
classify those with degree $s=2,3,4$ 
in $H(n,2)$, whose Hamming distances 
are symmetric with respect to $n/2$. 
In the appendix, we calculate the determinant of the second eigenmatrix of the fission schemes $\mathcal{S}$. 
This allows us to strengthen 
a theorem by Calderbank and Goethals 
\cite[Theorem~1]{CG}, who considered 
the existence of $(2s-2)$-designs with degree $s$ in the Hamming association schemes in the context of coding theory. 
(Note that $(2s-2)$-designs with degree $s$ in the Johnson association schemes, 
refereed to as {\it schematic designs}, 
were recently studied in \cite{BNS,KV}.)

It is known that tight Delsarte designs do not exist in most of the classical $P$- and $Q$-polynomial association schemes \cite{Chihara}. Thus, it is natural to quest for extremal designs in these schemes. It is shown in \cite{S,S2022} that the collection of derived designs of a $(2s-1)$-design with degree $s$ in an arbitrary $Q$-polynomial association scheme induces a coherent configuration. To study this configuration, it may be helpful to consider triple intersection numbers, which have proven useful in the studies on distance-regular graphs \cite{CJ}, $Q$-polynomial association schemes \cite{GVW}, and non-symmetric association schemes \cite{GavrilyukLansdownMunemasaSuda2025}. Furthermore, the proofs of Result~\ref{theo:ISineq} and Theorem~\ref{thm:e2} are based on regular semilattices, which also arise naturally in most of the classical schemes \cite{DELSARTE1976230}.

\section{Preliminaries}\label{sect:Definitions}

Here we recall some notions and facts needed in the subsequent sections.

\subsection{Association schemes}\label{subsect:AS}
We follow the standard notation 
and terminology from the theory 
of association schemes, see, e.g., 
\cite{BI}.
Let $X$ be a finite set of vertices
and a pair $(X,\{R_i\}_{i=0}^D)$ be
a symmetric \textbf{association scheme} of $D$ classes %(a $D$-class association scheme) 
defined on $X$ with a set of binary (symmetric) relations $\{R_0,R_1,\ldots,R_D\}$. 
Let $\mathcal{A}$ denote the Bose-Mesner 
algebra of the scheme\footnote{We will often refer to an association scheme simply as a ``scheme''.}, generated 
by the adjacency matrices $A_i\in \mathbb{R}^{X\times X}$ of 
the relations $R_i$
($0 \le i \le D$). Recall that
%\begin{enumerate}
%\item $A_0 =I_{|X|}$ is the identity matrix,
%\item $\sum_{i=0}^D A_i = J_{|X|}$ is the square all-one matrix,
%\item $A_i^\top=A_i$ ($1 \le i \le D$),
%\item 
\[A_iA_j=\sum_{k=0}^D p_{ij}^kA_k,\]
where $p_{ij}^k$ ($0 \le i,j,k \le D$) are nonnegative integers, called the \textbf{intersection numbers} of the scheme.
%\end{enumerate}
Let $E_0=\frac{1}{|X|}J_{|X|},E_1,\ldots,E_D$ 
denote the primitive idempotents 
of $\mathcal{A}$, which constitute 
another basis of the algebra, and recall that
%Since $\mathcal{A}$ is closed
%under the entry-wise multiplication $\circ$, it follows that 
\[
E_i\circ E_j=\frac{1}{|X|}\sum_{k=0}^D q_{ij}^kE_k,
\]
where $q_{ij}^k$ ($0 \le i,j,k \le D$) 
are nonnegative real numbers 
(see~\cite[Lemma~2.4]{D}), called 
the \textbf{Krein parameters} of the scheme.

Let $P$ and $Q$ denote the \textbf{first} and \textbf{second eigenmatrices} of the scheme, given by 
\[
(A_0, A_1, \ldots, A_d) = (E_0, E_1, \ldots, E_d)\cdot P\quad\text{and}\quad
Q=|X|\cdot P^{-1},
\]
and recall that the entries of $P$ are algebraic integers, as the $(j,i)$-entry of $P$ is an eigenvalue of $A_i$ for the eigenspace spanned by the columns of $E_j$.

A scheme 
is said to be \textbf{$Q$-polynomial} if, for some ordering of $E_1,\ldots,E_D$,
%for each $i$ ($0 \le i \le D$),
%there exists a polynomial $v_i^*(x)$ of %degree $i$
%such that $Q_{ji}=v_i^*(Q_{j1})$ ($0 \le j \le D$).
%It is also known that an association scheme is $Q$-polynomial
%if and only if 
the matrix $L_1^*:=(q_{1j}^k)_{k,j=0}^D$
of  Krein parameters
is tridiagonal with nonzero superdiagonal and subdiagonal~%
\cite[p.~193]{BI}: then $q_{ij}^k = 0$ holds whenever the triple $(i, j, k)$
does not satisfy the triangle inequality
(i.e., when $|i-j| < k$ or $i+j > k$).
For a $Q$-polynomial association scheme,
set $a_i^*=q_{1,i}^i$, $b_i^*=q_{1,i+1}^i$, and $c_i^*=q_{1,i-1}^i$
These numbers are usually gathered in the \textbf{Krein array}
$\{b_0^*, b_1^*, \dots, b_{D-1}^*; c_1^*, c_2^*, \dots, c_D^*\}$,
as the remaining Krein parameters 
can be computed from them.
A $Q$-polynomial scheme is called $Q$-\textbf{antipodal} if $b_i^*=c_{D-i}^*$ 
for $0\leq i\leq D$, $i\ne \lfloor D/2 \rfloor$.
By Van Dam \cite{vD99}, 3-class $Q$-antipodal schemes are equivalent to linked systems of symmetric designs, see \cite{Kodalen}.

Given an association scheme on $X$ of $D$ classes, 
a triple of vertices $u, v, w \in X$ 
and integers $i$, $j$, $k$ ($0 \le i, j, k \le D$), 
denote by $\sV{u & v & w}{i & j & k}$ (or simply $[i\ j\ k]$ 
when it is clear which triple $(u,v,w)$ we have in mind) 
the number of vertices $x \in X$ such that
$(u, x) \in R_i$, $(v, x) \in R_j$ and $(w, x) \in R_k$.
These numbers are usually referred to as \textbf{triple intersection numbers}. 
%Unlike the intersection numbers, 
The triple intersection numbers depend, 
in general, on the particular choice of $(u,v,w)$, and are not determined by the parameters (intersection numbers) of the scheme. 
Nevertheless, the following theorem 
often (especially, when the scheme is $Q$-polynomial) 
gives nontrivial equations with respect to triple intersection numbers.

\begin{theorem}{\rm (\cite[Theorem~3]{CJ}, 
cf.~\cite[Theorem~2.3.2]{BCN})}\label{thm:krein0}
Let $(X, \{R_i\}_{i=0}^D)$ be an association scheme of $D$ classes
with second eigenmatrix $Q$
and Krein parameters $q_{ij}^k$ $(0 \le i,j,k \le D)$.
Then
\[
%\pushQED{\bqed}
q_{ij}^k = 0 \quad \Longleftrightarrow \quad
\sum_{r,s,t=0}^D Q_{ri}Q_{sj}Q_{tk}\sV{u & v & w}{r & s & t} = 0
\quad \mbox{holds for all\ } u, v, w \in X. \qedhere
%\popQED
\]
\end{theorem}

These equations together with the standard double counting 
that relates triple intersection numbers to the ordinary intersection numbers 
(see, e.g., \cite{CJ}) may reveal the nonexistence of 
a putative association scheme or shed light on its structure. 
This was first observed in \cite{CGS} for schemes of 2 classes (i.e., strongly regular graphs) 
and later used to show the nonexistence of some putative 
distance-regular graphs (see, e.g.,~\cite{CJ,JV,U}) and $Q$-polynomial association schemes~\cite{GVW}; see also 
\cite{GavrilyukLansdownMunemasaSuda2025} for triple intersection numbers in non-symmetric association schemes.

Given a $Q$-polynomial association scheme on $X$ of $D$ classes, a subset $C$ of $X$ is a \textbf{$t$-design} if its characteristic vector 
$\chi:=\chi_C$ satisfies  
$E_i\chi=0$ for all $1 \le i \le t$ \cite{D}. A subset $C$ of $X$ is said to have \textbf{degree} $s$ if its characteristic vector 
$\chi$ satisfies  
$s=|\{ i \in\{1,\ldots,D\} \mid \chi^\top A_i \chi\neq 0\}|$, i.e., $C$ is an $s$-distance set.

Let $G\in \mathbb{R}^{X\times X}$ be 
a matrix that diagonalizes all the adjacency matrices of the scheme 
(see \cite[p.~11]{D}). We write $G$ in a column-partitioned form $G=(G_0\ G_1\ \cdots\ G_D)$ such that 
$GG^\top=|X|I_{|X|}$ 
and $E_i=\frac{1}{|X|}G_iG_i^\top$ ($0 \le i \le D$).
Define the $i$-th \textbf{characteristic matrix} $H_i=H_i(C)$ 
of a nonempty subset $C$ of $X$ as the submatrix of $G_i$ formed by the rows indexed by $C$. 
(Throughout this paper, a subset $C$ of $X$ is always nonempty.)

Let $(X,\{R_i\}_{i=0}^d)$ and $(X,\{S_i\}_{i=0}^e)$ be association schemes. 
Then the latter is called 
a \textbf{fission scheme} of the 
former if there exists a partition $\{\Lambda_0,\Lambda_1,\ldots,\Lambda_d\}$ of $\{0,1,\ldots,e\}$ such that $\Lambda_0=\{0\}$ and  $\cup_{k\in\Lambda_i}S_{k}=R_i$ for each $i$.

\subsection{Hamming association schemes and orthogonal arrays}\label{subsect:HammingOA}
Let $H(n,q)$ denote the \textbf{Hamming association scheme} defined on $X=[q]^n$, 
$[q]=\{1,2,\ldots,q\}$, $q\geq 2$, 
with $R_i:=\{(x,y)\mid x,y\in X, \partial(x,y)=i\}$ for $0\leq i\leq n$, where 
$\partial(x,y)$ denotes the \textbf{Hamming distance} between $x,y\in X$. 

The Hamming scheme $H(n,q)$ is $Q$-polynomial with the second eigenmatrix 
$Q=(K_{n,q,j}(i))_{i,j=0}^n$, 
where $K_{n,q,j}(x)$ are 
Krawtchouk polynomials 
of degree $j$ given by
\[
K_{n,q,j}(x)=\sum_{\ell=0}^j(-1)^\ell (q-1)^{j-\ell}\binom{x}{\ell}\binom{n-x}{j-\ell}.
\]

%(see the definition of $v_i^*(x)$ in Section \ref{subsect:AS}).
%Furthermore, the second eigenmatrix of $H(n,q)$ is $Q=(K_{n,q,j}(i))_{i,j=0}^n$.

Given an orthogonal array $\OA(N,n,q,t)$, 
the set $C$ of its $N$ rows can be naturally identified with a subset of the point set $X$ of $H(n,q)$. 
It follows from~\cite[Theorem~4.4]{D}
that $C$ is a $t$-design 
in the Hamming scheme $H(n,q)$ and vice versa. 
In what follows, we will identify an OA and the corresponding $t$-design $C$. 
%A fundamental problem is constructing orthogonal arrays with extremal parameters: given the strength $t$, the alphabet size $q$, and the number of columns $n$, we are interested in orthogonal arrays with minimum possible number of rows $N$.
%The lower bound on $N$ was given by Rao~\cite{R} as follows:
%\alg{this duplicates the introduction. Also $e$ is used instead of $s$.}
%\begin{numcases}{N\geq}
%    {\displaystyle \sum_{k=0}^e \binom{n}{k}(q-1)^k,\text{ if }t=2e,}\label{eq:tight}\\
%    {\displaystyle \sum_{k=0}^e \binom{n}{k}(q-1)^k+\binom{n-1}{e}(q-1)^{e+1},\text{ if }t=2e+1}.\label{eq:tighto}
%\end{numcases}
An orthogonal array is said to be \textbf{tight} if it satisfies Eq.\ \eqref{eq:tight} or \eqref{eq:tighto} with equality. 

Given a set $C\subseteq X$, define $C_i$ to be
\begin{equation}\label{eq:Ci}
C_i=\{(x_2,\ldots,x_n)\mid (i,x_2,\ldots,x_n)\in C\}
\quad (1 \le i \le q).   
\end{equation}
It is known that a $(2e+1)$-design $C$ is tight 
in $H(n,q)$ if and only if, for 
every $i\in\{1,2,\ldots,q\}$, 
the set $C_i$ is a 
tight $2e$-design in $H(n-1,q)$. 

The \textbf{degree set} of an orthogonal array $C$ is the set $S(C)$ of Hamming distances between pairwise distinct $x,y\in C$,
and the \textbf{degree} $s$ of $C$ is defined as $s=|S(C)|$. We also define the \textbf{degree set between $C$ and $C'$} as the set $S(C,C')$ of Hamming distances $\partial(x,y)$ for all $x\in C,y\in C'$, and the \textbf{degree} $s(C,C')$ between $C$ and $C'$ 
is defined as $s(C,C')=|S(C,C')|$. 
The following lemma gives 
an upper bound on 
the cardinality of an orthogonal array. 
\begin{lemma}{\rm (\cite[Theorem~5.21]{D})}
    Let $C$ be a subset of the point set of $H(n,q)$ with degree $s$. Then
    \begin{align}\label{eq:tightcode}
    |C|\leq \sum_{k=0}^s \binom{n}{k}(q-1)^k. 
    \end{align}
\end{lemma}
%\textcolor{red}{maybe we should mention here the result from the abstract (that degree is at least $t/2$). And then recall the definition of extremal OA.}
Recall that an OA is \textbf{extremal} if it is a $(2s-1)$-design in $H(n,q)$ with degree $s$ for some $s$. The following inequality is well known. 
\begin{corollary}
    Let $C$ be a %\textcolor{red}{tight?} 
    $t$-design in $H(n,q)$ with degree $s$. Then $t\leq 2s$. 
\end{corollary}
\begin{proof}
    Immediately by Eqs.\ \eqref{eq:tight}, \eqref{eq:tighto} 
    and \eqref{eq:tightcode}. 
%    \[
%\sum_{k=0}^{\lfloor t/2\rfloor} \binom{n}{k}(q-1)^k \leq \sum_{k=0}^s \binom{n}{k}(q-1)^k,
%    \]
%    which implies that $t\leq 2s$. 
\end{proof}

The following lemma characterizes designs in terms of their characteristic matrices $H_i$ 
(see the definition of $H_i$ in Section \ref{subsect:AS}).
The subsequent lemma and theorems can be formulated for any $Q$-polynomial association scheme, 
but we state them only for $H(n,q)$.
\begin{lemma}{\rm (\cite[Theorem~3.15]{D})}\label{lem:cha}
Let $C$ be a subset of the point set of $H(n,q)$. The following conditions are equi\-valent:
\begin{enumerate}
\item $C$ is a $t$-design,
\item $H_k^\top H_\ell=\delta_{k\ell}|C|I \quad
\text{for} \quad 0\leq k+\ell\leq t$,
\item $\sum_{x,y\in C}K_{n,q,i}(\partial(x,y))=0$ for any $i\in\{1,2,\ldots,t\}$.
\end{enumerate}
\end{lemma}

\begin{theorem}{\rm (\cite[Theorem~5.21]{D})}\label{thm:wilson} 
Let $C$ be a tight $t$-design in $H(n,q)$ with degree set $S=\{\alpha_1,\ldots,\alpha_s\}$,  
where $\alpha_1<\cdots<\alpha_s$.
\begin{enumerate}
\item If $t=2e$ then $s=e$ and $|C|\prod_{i=1}^e (1-x/\alpha_i)=\sum_{j=0}^e K_{n,q,j}(x)$ holds.
%In particular, $\sum_{j=0}^e K_{n,q,j}(x)$ has exactly $e$ distinct integral zeros in the interval $[1,n]$.
\item If $t=2e-1$ then $s=e$, $\alpha_s=n$ and $\frac{|C|}{q}\prod_{i=1}^{e-1} (1-x/\alpha_i)=\sum_{j=0}^{e-1} K_{n-1,q,j}(x)$ holds.
\end{enumerate}
\end{theorem}

The following theorem of Delsarte shows that 
a $t$-design $C$ in $H(n,q)$ with $t\geq 2s-2$ 
induces an $s$-class $Q$-polynomial association scheme (a {\bf Delsarte scheme}) 
whose binary relations are 
determined by the Hamming distances between its points. For a design $C$ with degree set 
$S(C)=\{\alpha_1,\ldots,\alpha_s\}$, 
we set $\alpha_0=0$ and define 
$S_i=\{(x,y)\in C\times C\mid \partial(x,y)=\alpha_i\}$ ($0 \le i \le s$).

\begin{theorem}{\rm (\cite[Theorem~5.25]{D})}\label{thm:t2s-2}
Let $C$ be a $t$-design in $H(n,q)$ with degree $s$ and degree set $S(C)=\{\alpha_1,\ldots,\alpha_s\}$.
If $t\geq 2s-2$, then the pair $(C,\{S_i\}_{i=0}^s)$ is a $Q$-polynomial association scheme of $s$ classes with the following second eigenmatrix:
\begin{align*}
Q=\begin{bmatrix}
1 & K_{n,q,1}(\alpha_0) & \cdots & K_{n,q,s-1}(\alpha_0) & |C|-\sum_{\ell=0}^{s-1}K_{n,q,\ell}(\alpha_0) \\
1 & K_{n,q,1}(\alpha_1) & \cdots & K_{n,q,s-1}(\alpha_1) & -\sum_{\ell=0}^{s-1}K_{n,q,\ell}(\alpha_1) \\
\vdots & \vdots & \ddots & \vdots& \vdots \\
1 & K_{n,q,1}(\alpha_s) & \cdots & K_{n,q,s-1}(\alpha_s) & -\sum_{\ell=0}^{s-1}K_{n,q,\ell}(\alpha_s)
\end{bmatrix}. 
\end{align*}
\end{theorem}

%The following theorem was shown in \cite{CG}.

%\begin{theorem}{\rm{\cite[Theorem~1]{CG}}}\label{thm:CG}
%If $C$ is a $(2s-2)$-design with degree $s$ and $S(C)=\{\alpha_1,\ldots,\alpha_s\}$, then 
%$\frac{q^{s(s-1)/2} \prod_{1\leq i<j \leq s}(\alpha_i-\alpha_j)}{\prod_{i=1}^{s-1}i!}$ 
%is an integer dividing $|C|^{s}$.
%\end{theorem}

%\alg{is this same as Theorem 4.3 below?Sho: Yes it is the same.}
%

%%%%%%%%%%%%%%%%%%%%%%%%%%%%%%%%%%%%%%%%%%%%%%%%%%%%%%%%%%%%%%%%%%%%%%%%%%%%%
%\section{$(2s-1)$-designs with degree $s$ in $H(n,q)$}
\section{Fission schemes for extremal orthogonal arrays}\label{sect:fission}

\subsection{Eigenmatrices}
According to Theorem~\ref{thm:t2s-2}, 
an extremal OA, that is, a $(2s-1)$-design with degree $s$ in $H(n,q)$ gives rise to a $Q$-polynomial association scheme of $s$ classes. 
In this section, we show that this scheme admits a fission scheme 
of at most $2s$ classes. 

Let $C$ be a $(2s-1)$-design in $H(n,q)$ with degree set $S(C)=\{\alpha_1,\ldots,\alpha_s\}$ and set $\alpha_0=0$.
%, and define $S_i=\{(x,y)\in C\times C\mid \partial(x,y)=\alpha_i\}$ for each $i$.
%By Theorem~\ref{thm:t2s-2},  the pair $(C,\{S_i\}_{i=0}^s)$ is an association scheme of $s$ classes.
%In this section, we decompose each $S_i$ into two subsets
%so that a $(2s-1)$-design with degree $s$ in $H(n,q)$
%yields an association scheme of \textcolor{red}{at most} $2s$ classes. 
%Recall that $C_i$ is 
%\[
%C_i=\{(x_2,\ldots,x_n)\mid (i,x_2,\ldots,x_n)\in C\}
%\quad (1 \le i \le q).
%\]
Recall the definition of $C_i$, see Eq. \eqref{eq:Ci}. 
Note that $C_i$ is obtained 
from $C$ by deleting the first coordinate of the vectors 
with $x_1=i$ in $C$ and $|C_i|=|C|/q$ for each $i$.
Then $C=\bigcup_{i=1}^q \{i\}\times C_i$ holds.
We will construct an association scheme on the point set $\widetilde{C}=\bigcup_{i=1}^{q}C_i$.

Denote by $H_k^{(i)}$ the $k$-th characteristic matrix of $C_i$ in $H(n-1,q)$,
and observe that $C_i$ is a $(2s-2)$-design with degree $s_i$ in $H(n-1,q)$, where $s_i=s(C_i)$.
First we state the following lemma, which is crucial for constructing our scheme on $\widetilde{C}$.

\begin{lemma}\label{lem:F}
Let $C$ be a $(2s-1)$-design with degree $s$ in $H(n,q)$.
Define $F_\ell^{(i,j)}$ to be
\begin{align*}
F_\ell^{(i,j)}=\frac{1}{\sqrt{|C_i||C_j|}}H_\ell^{(i)}(H_\ell^{(j)})^\top
\quad (1 \le i,j \le q, \ \ell\in\{0,1,\ldots,s-1\})
\end{align*}
and
\begin{align*}
F_s^{(i,i)} = I-\sum_{k=0}^{s-1}F_k^{(i,i)} \quad (1 \le i \le q).
\end{align*}
Then
$F_\ell^{(i,j)}F_{\ell'}^{(j,k)} = \delta_{\ell\ell'}F_\ell^{(i,k)}$
holds for $1 \le i,j,k \le q$ and $\ell,\ell'\in\{0,1,\ldots,s-1\}$,
and $F_{s}^{(i,i)}F_{\ell}^{(i,j)}$ $= F_{\ell}^{(i,j)}F_{s}^{(j,j)}$ $= O$
holds for $1 \le i,j \le q$ and $\ell\in\{0,1,\ldots,s-1\}$.
\end{lemma}
\begin{proof}
Apply Lemma~\ref{lem:cha}.
\end{proof}

Further, one can see that the degree set $S(\widetilde{C})$ is contained in 
$\{\alpha_i,\alpha_i-1\mid 1\leq i\leq s\}$ 
(as $\widetilde{C}=\bigcup_{i=1}^{q}C_i$ is a subset in $H(n-1,q)$).
Define $\widetilde{S}_{0,0},\widetilde{S}_{i,j}$ ($i\in\{0,1\},j\in\{1,2,\ldots,s\}$) by 
\begin{align*}
\widetilde{S}_{0,0}&=\{(x,y)\in\widetilde{C}\times \widetilde{C}\mid x,y\in C_k\text{ for some $k$ and } \partial(x,y)=0\},\\
\widetilde{S}_{0,j}&=\{(x,y)\in\widetilde{C}\times \widetilde{C}\mid x\in C_k,y\in C_\ell \text{ for some $k\ne\ell$, and }  \partial(x,y)=\alpha_{j}-1\},\\
\widetilde{S}_{1,j}&=\{(x,y)\in\widetilde{C}\times\widetilde{C}\mid  x,y\in C_k\text{ for some $k$, 
$x\ne y$
and }\partial(x,y)=\alpha_{j}\}.
\end{align*}
Note that $\widetilde{S}_{i,j}$ are symmetric relations on $\widetilde{C}$ that partition 
$\widetilde{C}\times \widetilde{C}$, 
and let $A_{i,j}$ denote the adjacency 
$(0,1)$-matrix of the 
binary relation $\widetilde{S}_{i,j}$.  %$(\widetilde{C},\widetilde{S}_{i,j})$. 
Further, define $A_\ell^{(i,j)}$ to be the submatrix of $A_{1,\ell}$ with the rows 
and columns restricted to $\widetilde{C}_i\times \widetilde{C}_j$ for distinct $i,j$.

Since, for $i\in \{1,2,\ldots,q\}$, 
$C_i$ is a $(2s-2)$-design 
with degree $s_i$ in $H(n-1,q)$, 
it follows that $s_i\leq s$; furthermore, $s-1\leq s_i$ holds 
by \cite[Theorem~5.21]{D}.  
Thus, $s_i\in \{s-1,s\}$. 
Set $s_{i,j}=s(C_i,C_j)$ % defined by
%$$
%s_{i,j}=|\{ \ell\in\{1,2,\ldots,s\} \mid \widetilde{S}_{1,\ell}\cap(C_i\times C_j)\neq \emptyset\}|
%$$
for distinct  $i,j\in\{1,2,\ldots,q\}$. 

\begin{proposition}
Let $C$ be a $(2s-1)$-design  with degree $s$ in $H(n,q)$. Then, with the above notation,  $s_{i,j}=s$ for all distinct $i,j$, and one of the following holds: 
  \begin{enumerate}
\item $s_i=s-1$ for every $i$ (and $C$ is a tight $(2s-1)$-design), 
\item $s_i=s$ for every $i$ (and $C$ is a nontight $(2s-1)$-design). 
\end{enumerate}
\end{proposition}
\begin{proof}
The case $s_i=s-1$ occurs if and only if $C_i$ is a tight $(2s-2)$-design 
in $H(n-1,q)$, that is, $C$ is a tight $(2s-1)$-design.  
%In this case, by $|C_i|=|C|/q$, any $C_j$ is also a tight $(2s-2)$-design in $H(n-1,q)$. 
Therefore, $s_i=s-1$ occurs for some $i$ if and only if $s_j=s-1$ for every $j$. 
Thus, we have the two cases from the statement. %consider the following two cases: 
%\begin{enumerate}
%\item $C$ is a tight $(2s-1)$-design and $s_i=s-1$ for every $i$, 
%\item $C$ is not a tight $(2s-1)$-design and $s_i=s$ for every $i$. 
%\end{enumerate}

%%Next, we define $s_{i,j}=s(C_i,C_j)$ % defined by
%$$
%s_{i,j}=|\{ \ell\in\{1,2,\ldots,s\} \mid \widetilde{S}_{1,\ell}\cap(C_i\times C_j)\neq \emptyset\}|
%$$
%%for distinct  $i,j\in\{1,2,\ldots,q\}$. 
Further, we claim that $s_{i,j}=s$ for all distinct $i,j$. 
For all distinct $i,j\in\{1,2,\ldots,q\}$, define 
\begin{align*}
\mathcal{A}^{(i,j)}=\text{span}\{A_{\ell}^{(i,j)} \mid 1\leq \ell \leq s\}. 
\end{align*} 
Then it is easy to see that $\dim \mathcal{A}^{(i,j)}=s_{i,j}$, and 
\begin{align*}
\mathcal{A}^{(i,j)}=\text{span}\{F_{\ell}^{(i,j)} \mid 0\leq \ell \leq s-1\}. 
\end{align*}
Note that the set $\{ F_{\ell}^{(i,j)} \mid \ell\in\{0,1,\ldots,s-1\}\}$ 
is linearly independent. 
Indeed, suppose 
$\sum_{\ell=0}^{s-1}c_\ell F_{\ell}^{(i,j)}=O$ for some scalars $c_\ell$. 
By Lemma \ref{lem:F}, multiplying the latter equality by 
$F_{m}^{(j,j)}$ for $m\in\{0,1,\ldots,s-1\}$ 
yields $c_m F_m^{(i,j)}=0$. 
Since each matrix $F_m^{(i,j)}$ is nonzero, $c_m=0$ holds. 
Therefore, $s_{i,j}=\dim \mathcal{A}^{(i,j)}=s$. 
\end{proof}

Define a matrix $M_{n-1,q}[x_1,\ldots,x_s]$ as follows:
\begin{align*}
M_{n-1,q}[x_1,\ldots,x_s]:=
\begin{bmatrix}
 K_{n-1,q,0}(x_1) &\cdots & K_{n-1,q,s-1}(x_1) \cr
 K_{n-1,q,0}(x_2) &\cdots & K_{n-1,q,s-1}(x_2) \cr
  \vdots &\ddots & \vdots  \cr
   K_{n-1,q,0}(x_s) &\cdots & K_{n-1,q,s-1}(x_s) 
\end{bmatrix}.
\end{align*}
The next theorem is main in this section. 
It shows that a $(2s-1)$-design $C$ 
with degree $s$ 
in $H(n,q)$ gives rise to an association scheme of $2s-1$ or $2s$ classes, which is the fission scheme 
of the Delsarte scheme of $C$.
%such a scheme will be referred to as the {\em fission scheme} of $C$. 
\begin{theorem}\label{thm:qant4}
Let $C$ be a $(2s-1)$-design with degree $s$ in $H(n,q)$.
\begin{enumerate}
    \item If the design $C$ is tight, then $(\widetilde{C},\{\widetilde{S}_{0,0}\}\cup\{\widetilde{S}_{i,j} \mid i\in\{0,1\},j\in\{1,2,\ldots,s\}\}\setminus \{\widetilde{S}_{0,s}\})$
is an association scheme of $(2s-1)$ classes with the second eigenmatrix 
\begin{comment}
\begin{align*}
    \hspace{-1cm}\left(\begin{smallmatrix}
 K_{n-1,q,0}(\alpha_0) &\cdots & K_{n-1,q,s-1}(\alpha_0)  & (q-1)K_{n-1,q,0}(\alpha_0) &\cdots & (q-1)K_{n-1,q,s-1}(\alpha_0) \cr
 K_{n-1,q,0}(\alpha_1) &\cdots & K_{n-1,q,s-1}(\alpha_1) & (q-1)K_{n-1,q,0}(\alpha_1) &\cdots & (q-1)K_{n-1,q,s-1}(\alpha_1) \cr
  \vdots &\ddots & \vdots & \vdots & \vdots \cr
   K_{n-1,q,0}(\alpha_{s-1}) &\cdots & K_{n-1,q,s-1}(\alpha_{s-1}) & (q-1)K_{n-1,q,0}(\alpha_{s-1}) &\cdots & (q-1)K_{n-1,q,s-1}(\alpha_{s-1}) \cr
 K_{n-1,q,0}(\alpha_{1}-1) &\cdots & K_{n-1,q,s-1}(\alpha_{1}-1)&  -K_{n-1,q,0}(\alpha_{1}-1) & \cdots & -K_{n-1,q,s-1}(\alpha_{1}-1) \cr
 \vdots &\ddots & \vdots & \vdots&  \ddots & \vdots \cr
 K_{n-1,q,0}(\alpha_{s}-1) &\cdots & K_{n-1,q,s-1}(\alpha_{s}-1)&  -K_{n-1,q,0}(\alpha_{s}-1) & \cdots & -K_{n-1,q,s-1}(\alpha_{s}-1) 
\end{smallmatrix}\right). 
\end{align*}
\end{comment}
\begin{align*}
Q=
\begin{bmatrix}
 M_{n-1,q}[\alpha_0,\ldots,\alpha_{s-1}] & (q-1)M_{n-1,q}[\alpha_0,\ldots,\alpha_{s-1}] \cr
 M_{n-1,q}[\alpha_1-1,\ldots,\alpha_{s}-1] & (-1)M_{n-1,q}[\alpha_1-1,\ldots,\alpha_{s}-1] 
\end{bmatrix}. 
\end{align*}
\item If the design $C$ is not tight, then $(\widetilde{C},\{\widetilde{S}_{0,0}\}\cup\{\widetilde{S}_{i,j} \mid i\in\{0,1\},j\in\{1,2,\ldots,s\}\})$
is an association scheme of $2s$ classes with the second eigenmatrix  
\begin{comment}
\begin{align*}
    \hspace{-1cm}\left(\begin{smallmatrix}
 K_{n-1,q,0}(\alpha_0) &\cdots & K_{n-1,q,s-1}(\alpha_0) & |C|-q\sum_{\ell=0}^{s-1}K_{n-1,q,\ell}(\alpha_0) & (q-1)K_{n-1,q,0}(\alpha_0) &\cdots & (q-1)K_{n-1,q,s-1}(\alpha_0) \cr
 K_{n-1,q,0}(\alpha_1) &\cdots & K_{n-1,q,s-1}(\alpha_1)& -q\sum_{\ell=0}^{s-1}K_{n-1,q,\ell}(\alpha_1) & (q-1)K_{n-1,q,0}(\alpha_1) &\cdots & (q-1)K_{n-1,q,s-1}(\alpha_1) \cr
  \vdots &\ddots & \vdots & \vdots & \vdots \cr
   K_{n-1,q,0}(\alpha_s) &\cdots & K_{n-1,q,s-1}(\alpha_s)& -q\sum_{\ell=0}^{s-1}K_{n-1,q,\ell}(\alpha_s) & (q-1)K_{n-1,q,0}(\alpha_s) &\cdots & (q-1)K_{n-1,q,s-1}(\alpha_s) \cr
 K_{n-1,q,0}(\alpha_{1}-1) &\cdots & K_{n-1,q,s-1}(\alpha_{1}-1)& 0& -K_{n-1,q,0}(\alpha_{1}-1) & \cdots & -K_{n-1,q,s-1}(\alpha_{1}-1) \cr
 \vdots &\ddots & \vdots & \vdots& \vdots & \ddots & \vdots \cr
 K_{n-1,q,0}(\alpha_{s}-1) &\cdots & K_{n-1,q,s-1}(\alpha_{s}-1)& 0& -K_{n-1,q,0}(\alpha_{s}-1) & \cdots & -K_{n-1,q,s-1}(\alpha_{s}-1) 
\end{smallmatrix}\right). 
\end{align*}
\end{comment}
\begin{align*}
Q=
\begin{bmatrix}
 M_{n-1,q}[\alpha_0,\ldots,\alpha_{s}] & \mathbf{c} & (q-1)M_{n-1,q}[\alpha_0,\ldots,\alpha_{s}] \cr
 M_{n-1,q}[\alpha_1-1,\ldots,\alpha_{s}-1] & O_{s\times 1} & (-1)M_{n-1,q}[\alpha_1-1,\ldots,\alpha_{s}-1] 
\end{bmatrix}, 
\end{align*}
where $\mathbf{c}=\begin{bmatrix}
 |C|-q\sum_{\ell=0}^{s-1}K_{n-1,q,\ell}(\alpha_0) \cr
 -q\sum_{\ell=0}^{s-1}K_{n-1,q,\ell}(\alpha_1) \cr
 \vdots \cr
 -q\sum_{\ell=0}^{s-1}K_{n-1,q,\ell}(\alpha_s)
\end{bmatrix}$.

\begin{comment}
\begin{align*}
Q=\begin{bmatrix}
 K_{n-1,q,j}(\alpha_0) & |C|-q\sum_{\ell=0}^{s-1}K_{n-1,q,\ell}(\alpha_0) & (q-1)K_{n-1,q,j'}(\alpha_0) \\
 K_{n-1,q,j}(\alpha_i) & -q\sum_{\ell=0}^{s-1}K_{n-1,q,\ell}(\alpha_i) & (q-1)K_{n-1,q,j'}(\alpha_i) \\
 K_{n-1,q,j}(\alpha_{i'}-1) & 0& -K_{n-1,q,j'}(\alpha_{i'}-1) 
\end{bmatrix}
\end{align*}
where $i,i'\in\{1,\ldots,s\},j,j'\in\{0,1,\ldots,s-1\}$. 
\end{comment}
\end{enumerate}
\end{theorem}
\begin{proof}
It follows from Theorem~\ref{thm:t2s-2} that each $C_i$, being a $(2s-2)$-design 
with degree $s_i$ in $H(n-1,q)$, defines an association scheme of $s_i$ classes, 
whose relations are determined by the Hamming distances from the degree set $S(C_i)$.
Furthermore, it follows from the proof of Theorem~\ref{thm:t2s-2} that 
the primitive idempotents of this scheme are defined in Lemma \ref{lem:F} as follows:
\[
F_{0}^{(i,i)},F_{1}^{(i,i)},\ldots,F_{s-1}^{(i,i)},\quad\text{and, if $s_i=s$,}\quad
F_{s}^{(i,i)}=I-\sum_{k=0}^{s-1}F_{k}^{(i,i)},
\]
while in the case $s_i=s-1$ the matrix $F_s^{(i,i)}$ is zero (and hence 
not an idempotent).
%where $F_{\ell}^{(i,i)}=\frac{1}{|C_i|}H_\ell^{(i)}(H_\ell^{(i)})^\top$ for $\ell\in\{0,1,\ldots,s-1\}$. 

Now we define $E_{0,0},E_{0,1},\ldots,E_{0,s},E_{1,0},\ldots,E_{1,s-1}$ as follows:
\begin{align*}
E_{0,i}&=\frac{1}{q}\begin{bmatrix}
F_{i}^{(1,1)} & F_{i}^{(1,2)} & \cdots & F_{i}^{(1,q)} \\
F_{i}^{(2,1)} & F_{i}^{(2,2)} & \cdots & F_{i}^{(2,q)} \\
\vdots & \vdots & \ddots & \vdots \\
F_{i}^{(q,1)} & F_{i}^{(q,2)} & \cdots & F_{i}^{(q,q)}
\end{bmatrix} \text{ for } 
0\leq i\leq s-1,\\
%i\in\{0,1,\ldots,s-1\},\\
E_{0,s}&=\phantom{\frac{q}{q}}
\begin{bmatrix}
F_{s}^{(1,1)} & O & \cdots & O \\
O & F_{s}^{(2,2)} & \cdots & O \\
\vdots & \vdots & \ddots & \vdots \\
O & O & \cdots & F_{s}^{(q,q)}
\end{bmatrix},\\
E_{1,s-1-i}&=\frac{1}{q}\begin{bmatrix}
(q-1)F_{i}^{(1,1)} & -F_{i}^{(1,2)} & \cdots & -F_{i}^{(1,q)} \\
-F_{i}^{(2,1)} & (q-1)F_{i}^{(2,2)} & \cdots & -F_{i}^{(2,q)} \\
\vdots & \vdots & \ddots & \vdots \\
-F_{i}^{(q,1)} & -F_{i}^{(q,2)} & \cdots & (q-1)F_{i}^{(q,q)}
\end{bmatrix} \text{ for } 
0\leq i\leq s-1.
%i\in\{0,1,\ldots,s-1\}.
\end{align*}

Note that each $E_{i,j}$, possibly except for $E_{0,s}$, is a nonzero matrix, and $E_{0,s}$ is a zero matrix if and only if the degree of $C_i$ is $s-1$.

Recall that $A_{i,j}$ is the adjacency matrix of the binary relation $\widetilde{S}_{i,j}$ on $\widetilde{C}$, 
for $(i,j)\in \{(k,\ell)\mid 0\leq k\leq 1,1 \le \ell \le s\}$, and let $\mathcal{A}$ denote the vector space spanned by all $A_{i,j}$ over $\mathbb{R}$.
Since the matrices
\begin{align*}
\begin{bmatrix}
F_{i}^{(1,1)} & O & \cdots & O \\
O & F_{i}^{(2,2)} & \cdots & O \\
\vdots & \vdots & \ddots & \vdots \\
O & O & \cdots & F_{i}^{(q,q)}
\end{bmatrix},\quad
\begin{bmatrix}
O & F_{i}^{(1,2)} & \cdots & F_{i}^{(1,q)} \\
F_{i}^{(2,1)} & O & \cdots & F_{i}^{(2,q)} \\
\vdots & \vdots & \ddots & \vdots \\
F_{i}^{(q,1)} & F_{i}^{(q,2)} & \cdots & O
\end{bmatrix}
\end{align*}
are written as linear combinations of $A_{0,0},A_{0,1},\ldots,A_{0,s},A_{1,1},\ldots,A_{1,s}$, so are the matrices $E_{0,0},E_{0,1},\ldots,E_{0,s},E_{1,0},\ldots,E_{1,s-1}$. 
From Lemma~\ref{lem:F}, it follows that the nonzero matrices from $E_{0,0},E_{0,1},\ldots,E_{0,s}$, $E_{1,0},\ldots,E_{1,s-1}$ are 
mutually orthogonal idempotents.
Thus, $\mathcal{A}$ is closed under the matrix multiplication, and 
the set $\widetilde{C}$ with the binary relations as defined in the statement 
of the theorem form an association scheme %of $2s$ or $(2s-1)$ classes 
with primitive idempotents 
$E_{0,0},E_{0,1},\ldots,E_{0,s-1},E_{1,0},\ldots,E_{1,s-1}$ and 
with $E_{0,s}$ if $s_i=s$.

Next, we determine the second eigenmatrix $Q$ of this scheme. 
\begin{enumerate}
\item 
If $s_i=s-1$, then $Q$ is defined by the following blocks: 
\begin{align*}
\bbordermatrix{& E_{0,j}  & E_{1,j'} \cr
A_{0,i} & K_{n-1,q,j}(\alpha_i) & (q-1)K_{n-1,q,j'}(\alpha_i) \cr
A_{1,i'} & K_{n-1,q,j}(\alpha_{i'}-1) & -K_{n-1,q,j'}(\alpha_{i'}-1) }
\end{align*}
where $i\in\{0,1,\ldots,s-1\},i'\in\{1,\ldots,s\},j,j'\in\{0,1,\ldots,s-1\}$ and $S(C_i)=\{\alpha_1,\ldots,\alpha_{s-1}\}$, $S(C_i,C_j)=\{\alpha_1-1,\ldots,\alpha_s-1\}$, where $\alpha_s=n$. 
\item 
If $s_i=s$, then $Q$ is defined by the following blocks: 
\begin{align*}
\bbordermatrix{& E_{0,j}  & E_{0,s} & E_{1,j'} \cr
A_{0,0} & K_{n-1,q,j}(\alpha_0) & |C|-q\sum_{\ell=0}^{s-1}K_{n-1,q,\ell}(\alpha_0) & (q-1)K_{n-1,q,j'}(\alpha_0) \cr
A_{0,i} & K_{n-1,q,j}(\alpha_i) & -q\sum_{\ell=0}^{s-1}K_{n-1,q,\ell}(\alpha_i) & (q-1)K_{n-1,q,j'}(\alpha_i) \cr
A_{1,i'} & K_{n-1,q,j}(\alpha_{i'}-1) & 0& -K_{n-1,q,j'}(\alpha_{i'}-1) }
\end{align*}
where $i,i'\in\{1,\ldots,s\},j,j'\in\{0,1,\ldots,s-1\}$ and $S(C_i)=\{\alpha_1,\ldots,\alpha_s\}$, $\alpha_0=0$, $S(C_i,C_j)=\{\alpha_1-1,\ldots,\alpha_s-1\}$. 
\end{enumerate}
The lemma is proved.
\end{proof}

\subsection{Application: tight $3$-designs}\label{sect:tight3}
In this section, we apply Theorem~\ref{thm:qant4} in the case $s=2$. 
Assume that there exists a $3$-design $C$ with degree $2$ in $H(n,q)$. 
Each derived design $C_i$ is a $2$-design with degree at most $2$ in $H(n-1,q)$. 
The fission scheme of the Delsarte scheme of $C$, 
constructed by Theorem~\ref{thm:qant4}, 
is of 
$3$ classes or $4$ classes depending 
on whether $C$ is tight or not. We will discuss the nontight case in the Appendix \ref{app:A}.

Suppose now that the design $C$ is 
tight and set $S(C)=\{\alpha_1,\alpha_2\}$ and  $\alpha_0=0$.  
%Assume the Hamming distance $\alpha_1$ appears in each $C_i$ \textcolor{red}{(???)}. 
Then the $3$-class fission scheme of $C$ has the second eigenmatrix $Q$ given by: 
\begin{align*}
Q=\bbordermatrix{& E_{0,0}  & E_{0,1}  & E_{1,0} & E_{1,1} \cr
A_{0,0} & 1 & K_{n-1,q,1}(\alpha_0) & (q-1)K_{n-1,q,1}(\alpha_0) & q-1 \cr
A_{0,1} & 1 & K_{n-1,q,1}(\alpha_1) & (q-1)K_{n-1,q,1}(\alpha_1) & q-1 \cr
A_{1,1} & 1 & K_{n-1,q,1}(\alpha_1-1)  & -K_{n-1,q,1}(\alpha_1-1) & -1 \cr
A_{1,2} & 1 & K_{n-1,q,1}(\alpha_2-1)  & -K_{n-1,q,1}(\alpha_2-1) & -1 }.
\end{align*} 

It is shown in \cite{N1986} that $(|C|,n,q,\alpha_1,\alpha_2)$ is $(2n,n,2,n/2,n)$ 
with $n\equiv0\pmod{4}$ or $(q^3,q+2,q,q,q+2)$ with $q$ even. 
In the former case, the existence 
of $C$ 
is equivalent to that of a Hadamard matrix 
of order $n$; the 3-class fission scheme 
corresponds to a distance-regular graph 
defined by the Hadamard matrix, see \cite[Theorem 1.6.1]{BCN}.
%a distance-regular Hadamard graph, 
%see \cite[Theorem 1.6.1\textcolor{red}{maybe Cor. 1.8.2?}]{BCN}.}

In the latter case, a construction is known \cite[Section~5]{BB} 
for every $q$ being a power of $2$, in which case each $C_i$ is a complete set of mutually orthogonal Latin squares of order $q$. 

%\subsection{$(|C|,n,q)=(q^3,q+2,2)$}
%Next 
Consider further the case $(|C|,n,q)=(q^3,q+2,q)$ with $q$ even. 
We have $\alpha_1=q,\alpha_2=q+2$, 
and the $3$-class fission scheme of $C$ has the 
following second eigenmatrix $Q$ and matrix $L_1^*$ 
(see the definition of $L_1^*$ in Section \ref{subsect:AS}): 
\begin{align}
Q&=\left[
\begin{array}{cccc}
 1 & q^2-1 & (q-1)^2 (q+1) & q-1 \\
 1 & -1 & 1-q & q-1 \\
 1 & q-1 & 1-q & -1 \\
 1 & -q-1 & q+1 & -1 \\
\end{array}
\right],\label{eq:eigen} \\
L_1^*&=\left[
\begin{array}{cccc}
 0 & q^2-1 & 0 & 0 \\
 1 & q-2 & (q-1) q & 0 \\
 0 & q & (q-2) (q+1) & 1 \\
 0 & 0 & q^2-1 & 0 \\
\end{array}\nonumber
\right],
\end{align}
which means that the fission scheme is $Q$-polynomial and has the Krein array 
$\{q^2-1,(q-1)q,1;1,q,q^2-1\}$. 

\begin{theorem}\label{thm:main}
Suppose there exists a $Q$-polynomial association scheme with the Krein array 
$\{q^2-1,(q-1)q,1;
1,q,q^2-1\}$. 
If $q>2$, then $q\equiv 0,1\pmod{4}$. 
\end{theorem} 
\begin{proof}
Assume $q>2$. 
Computing the intersection numbers shows that $p_{2,2}^2=\frac{1}{4}q(q+3)(q-2)>0$; hence there exist three points of the scheme in pairwise relation $R_2$. 
Computing (using, e.g., the sage-drg package \cite{JV}) 
the triple intersection numbers 
with respect to these three points, we get 
$[1,2,3]=\frac{1}{4}(q^2-q)$, 
which thus must be an integer. 
Therefore, $q\equiv 0,1\pmod{4}$ as desired. 
%Since $q$ is even, $q$ must be a multiple of four.  
\end{proof}
\begin{corollary}\label{cor:tight3}
Suppose there exists a tight $3$-design in $H(q+2,q)$. 
If $q>2$, then $q$ is a multiple of four. 
\end{corollary} 
\begin{proof}
By Theorem~\ref{thm:qant4}, 
a tight $3$-design $C$ in $H(q+2,q)$
gives rise to the fission scheme 
with the second eigenmatrix in Eq. \eqref{eq:eigen}. 
Since $q$ is even by the result of \cite{N1986}, $q$ must be a multiple of four by Theorem~\ref{thm:main}. 
\end{proof}

\begin{remark}
    The scheme from Theorem \ref{thm:main} corresponds to a linked system of symmetric designs with parameters $v=q^2$, $k=q(q-1)/2$, $\lambda=q(q-2)/4$ by \cite{vD99}. The examples are known for $q=2t$, see Manon's family in \cite{Kodalen}. The linked systems of symmetric designs is of ``optimistic" type in the sense of \cite{Kodalen}, and \cite[Theorem~7.5]{Kodalen} shows that $q$ is a multiple of $4$, which is stronger than Theorem~\ref{thm:main}.  
     
\end{remark}

\section{Hamming distances in extremal orthogonal arrays}\label{Section4}

\subsection{Inequalities}\label{sect:ISineq}
In this section, we prove the following theorem, which can be seen 
as a counterpart of Result 
\ref{theo:ISineq} for extremal 
combinatorial designs.

\begin{theorem}\label{thm:e2}
Let $C$ be a $(2s-1)$-design in $H(n,q)$ with degree $s$. %with index $\lambda_{2s-1}$, where $\lambda_{2s-1}=\frac{|C|}{q^{2s-1}}$. 
%Regarding $C$ as an $i$-design, its index $\lambda_i$ is $\lambda_i=\frac{|C|}{q^i}$. 
Define %$S(C)=\{\partial(x,y)\mid x,y\in C,x \neq y\}$ and 
$S'(C)=\{n-\partial(x,y)\mid x,y\in C,x \neq y\}$ and set $S'(C)=\{x_1,\ldots,x_s\}$ with $x_1<\cdots<x_s$. 
Then 
\begin{align*}
\frac{(s-1)(n-s)}{q}+\frac{s(s-1)}{2}\leq \sum_{i=1}^sx_i\leq \frac{s(n-s)}{q}+\frac{s(s-1)}{2},     
\end{align*}
with equality in the left if and only if $x_1=0$, and with equality in the right if and only if $C$ is a tight $2s$-design.  
\end{theorem}

\begin{example}\label{extremalOA}
Consider the following two well-known 
designs.
\begin{enumerate}
    \item 
    Let $C$ be the dual of the 
    %\textcolor{red}{the dual code of the extended ternary Golay code?}
    extended ternary Golay code in $H(12,3)$. 
    Then $S'(C)=\{x_1=0,x_2=3,x_3=6\}$, and the left-hand side in Theorem~\ref{thm:e2} equals
    \[
    \frac{(s-1)(n-s)}{q}+\frac{s(s-1)}{2}=\frac{(3-1)(12-3)}{3}+\frac{3(3-1)}{2}=9=x_1+x_2+x_3.
    \]
\item Let $C$ be the dual code of the ternary Golay code in $H(11,3)$, which is a tight $4$-design. 
    Then $S'(C)=\{x_1=2,x_2=5\}$, and the right-hand side in Theorem~\ref{thm:e2} equals
    \[
    \frac{s(n-s)}{q}+\frac{s(s-1)}{2}=\frac{2(11-2)}{3}+\frac{2(2-1)}{2}=7=x_1+x_2.
    \]
\end{enumerate}
This shows that both sides 
of the inequality are tight.
    \end{example}

We will need a series of lemmas 
for the proof of Theorem \ref{thm:e2}.
For a set $C$ of the point set of the Hamming scheme $H(n,q)$, recall $S(C)=\{\partial(x,y)\mid x,y\in C,x \neq y\}$ and $S'(C)=\{n-\partial(x,y)\mid x,y\in C,x \neq y\}$.  
Let $S'(C)=\{x_1,\ldots,x_s\}$ with $x_1<\cdots<x_s$.

Following Ionin and Shrikhande \cite{IS},
for pairwise distinct nonnegative integers $x_1,x_2,\ldots, x_s$, we define 
the following polynomial:
\[
P_s(z)=\prod_{i=1}^s (z-x_i),
\]
and the following recurrence sequence
$F_j^{(k)}$ ($k\geq 1, 0\leq j \leq k\}$) by  $F_0^{(k)}=1,F_k^{(k)}=x_1\cdots x_k$ and 
\begin{align*}
    F_j^{(k)}=F_j^{(k-1)}+(x_k-k+j)F_{j-1}^{(k-1)}\quad (k\geq 2, 1\leq j\leq k). 
\end{align*}

\begin{lemma}\label{lem:e1-4}{\rm (\cite[Lemma~1.4]{IS})} 
    $P_s(z)=\sum_{j=0}^s (-1)^{j}F_{j}^{(s)}\cdot (z)_{s-j}$, where 
    $(z)_i$ is the Pochhammer symbol
    $z(z+1)\cdots(z+i-1)$.
\end{lemma}

Define the $(s+1)\times (s+1)$ matrix   $M_s=M_s(x_1,\ldots,x_s;a_0,a_1,\ldots,a_s)$ by 
\begin{align*}
    M_s=\begin{bmatrix}
        a_0 & 1 & 1 & \cdots & 1 \\
        a_1 & x_1 & x_2 & \cdots & x_s \\
        a_2 & (x_1)_2 & (x_2)_2 & \cdots & (x_s)_2 \\
        \vdots & \vdots & \vdots & \ddots & \vdots \\
        a_s & (x_1)_s & (x_2)_s & \cdots & (x_s)_s
    \end{bmatrix}. 
\end{align*}

\begin{lemma}\label{lem:detMs} 
The following holds.
\begin{enumerate}
    \item[(1)] 
    $
    \det M_s(x_1-1,\ldots,x_s-1;a_0,a_1,\ldots,a_s)=\det M_s(x_1,\ldots,x_s;b_0,b_1,\ldots,b_s),
    $
    where $b_0=a_0$ and $b_i=i a_{i-1}+a_{i}$ for $i\in\{1,\ldots,s\}$. 
    \item[(2)] {\rm (\cite[Proposition~2.2]{IS})} The determinant of $M_s$ is equal to
    \[
    (-1)^s\prod_{1\leq i<j\leq s} (x_j-x_i)\sum_{j=0}^s (-1)^ja_{s-j}F_j^{(s)}.
    \]
\end{enumerate}
%One has:
\end{lemma}
\begin{proof}
Let $A=M_s(x_1-1,\ldots,x_s-1;a_0,a_1,\ldots,a_s)$. 
    Note that $j(x_i-1)_{j-1}+(x_i-1)_{j}=(x_i)_{j}$. 
    For $i=s,s-1,\ldots,1$, adding $i$ times the $(i-1)$-th row of $A$ to the $i$-th row of $A$ yields the other matrix.  
\end{proof}

Define a semilattice attached to 
the Hamming scheme $H(n,q)$ as follows. 
Let $n,q$ be integers such that $n\geq 1, q\geq 2$.
Let $V$ and $F$ be finite sets with sizes $n$ and $q$ respectively.
Define 
$X=\{(E,f): E\subseteq V, f\in\text{Map}(E,F)\}$.
For $(E,f),(E',f')\in X$, set $(E,f)\preceq (E',f')$ if and only if $E\subseteq E'$ and $f'|_E=f$.
Then $(X,\preceq)$ forms a semilattice with rank function, denoted $|x|$ for $x\in X$, taking the size of $E$. 
For every $0\leq i\leq n$, define the fiber $X_i:=\{x\in X: |x|=i\}$.
We regard $X_n$ as the vertex set of the Hamming scheme $H(n,q)$. 
As usual, the element $x\wedge y$ denotes the greatest lower bound of $x$ and $y$ for $x,y\in X$. 

Recall that the \textbf{index} $\lambda:=\frac{N}{q^t}$ of an $\OA(N,n,q,t)$ 
is the multiplicity of any 
row of $[q]^t$ among any chosen $t$ columns 
of the array.
Let $C$ be a $(2s-1)$-design in $H(n,q)$ with index $\lambda_{2s-1}$, where $\lambda_{2s-1}=\frac{|C|}{q^{2s-1}}$. Regarding $C$ as an $i$-design, 
for $i\leq 2s-1$, 
its index $\lambda_i$ is $\lambda_i=\frac{|C|}{q^i}$. 

\begin{lemma}\label{lem:e11}
Let $C$ be a $t$-design in $H(n,q)$ with degree $s$. 
\begin{enumerate}
    \item If $t\geq s$, then 
    \begin{align}
    \sum_{j=0}^s (-1)^j (n)_{s-j}\lambda_{s-j}F_{j}^{(s)}=P_s(n).\label{eq:e2}
    \end{align}
    \item If $t\geq 2s-1$, then
\end{enumerate}
    \begin{align}
\sum_{j=0}^s (-1)^j (n-\ell)_{s-j}\lambda_{s-j}F_{j}^{(s)}=0 \quad \text{ for } 1\leq \ell \leq s-1.\label{eq:e3}
    \end{align}
\end{lemma}
\begin{proof}
(1) Let $y$ be a fixed point %codeword 
in $C$. 
    For $j\in\{1,\ldots,s\}$, let $m_j$ denote the number of points
    %codewords 
    $x$ such that $n-\partial(x,y)=x_j$. 
    Additionally, set $m_0=-1$ and $x_0=n$.
    Then, for $i\in\{0,1,\ldots,s\}$, double counting 
    the set 
    \[
    \{(x,I)\in C\times X_{i} \mid I\preceq x\wedge y\}
    \]
    yields 
    \begin{align}\label{eq:e1}
    \sum_{j=0}^s m_j \binom{x_j}{i}=\binom{n}{i}\lambda_i.
    \end{align}
    Multiplying the $i$-th equation by $i!$, one can consider 
    Eq.~\eqref{eq:e1} with $i\in\{0,1,\ldots,s\}$ 
    as a homogeneous system of linear equations in the 
    unknowns $m_0,m_1,\ldots,m_s$ with the coefficient matrix $M_s=M_s(x_1,\ldots,x_s;a_0,a_1,\ldots,a_s)$, where
    $a_i=(n)_i (\lambda_i-1)$ for $i\in\{0,1,\ldots,s\}$. 
    As we set $m_0\ne 0$, it follows that $\det M_s=0$. 
    By %\cite[Proposition~2.2]{IS}, 
    Lemma \ref{lem:detMs}, 
    \[
    \sum_{j=0}^s (-1)^ja_{s-j}F_j^{(s)}= 
    \sum_{j=0}^s (-1)^j(n)_{s-j}(\lambda_{s-j}-1)F_j^{(s)}=0, 
    \]
    which, together with  Lemma~\ref{lem:e1-4}, gives Eq. ~\eqref{eq:e2}. This shows (1). 

    (2) Next, consider the contraction 
    (say, by the first coordinate) 
    $C_1$ of $C$, which is a $(2s-2)$-design with index $\lambda'_{2s-2}$ in $H(n-1,q)$. Note that $S'(C_1)$ is contained in $\{x_1-1,\ldots,x_s-1\}$ and $\lambda_i'=\lambda_{i+1}$. 
    Lemma \ref{lem:detMs} and the above argument\footnote{When $x_i-1\not\in S'(X)$, we set $m_i=0$. Since we set $m_0=-1$, the homogeneous system of linear equations whose unknowns are $m_0,m_1,\ldots,m_s$ has a nontrivial solution.} applied to $C_1$ 
    yield 
    \[
    \det M_s(x_1,\ldots,x_s;b_0,b_1,\ldots,b_s)=\det M_s(x_1-1,\ldots,x_s-1;a_0,a_1,\ldots,a_s)=0, 
    \]
    where $a_i=(n-1)_i(\lambda'_{i}-1)$ for $i\in\{0,1,\ldots,s\}$, $b_0=a_0$ and $b_i=ia_{i-1}+a_i$ for $i\in\{1,\ldots,s\}$.  
    Then: 
    \begin{small}%
    \begin{align*}
    0&=\sum_{j=0}^s (-1)^j b_{s-j}F_{j}^{(s)}\\
    &=\sum_{j=0}^s (-1)^j F_{j}^{(s)} a_{s-j}+\sum_{j=0}^{s} (-1)^j F_{j}^{(s)} (s-j)a_{s-j-1}\displaybreak[0]\\
    &=\sum_{j=0}^s (-1)^j F_{j}^{(s)} (n-1)_{s-j}(\lambda'_{s-j}-1)+\sum_{j=0}^{s} (-1)^j F_{j}^{(s)} ((n)_{s-j}-(n-1)_{s-j})(\lambda'_{s-j-1}-1)\displaybreak[0]\\
    %&=\sum_{j=0}^s (-1)^j F_{j}^{(s)} (n-1)_{s-j}(\lambda'_{s-j}-\lambda'_{s-j-1})+\sum_{j=0}^{s} (-1)^j F_{j}^{(s)} (n)_{s-j}\lambda'_{s-j-1}-\sum_{j=0}^{s} (-1)^j F_{j}^{(s)} (n)_{s-j}\displaybreak[0]\\
    &=\sum_{j=0}^s (-1)^j F_{j}^{(s)} (n-1)_{s-j}(\lambda_{s-j+1}-\lambda_{s-j})+P_s(n)-P_s(n)\displaybreak[0]\\
    &\text{ (by Lemma~\ref{lem:e11}(1) and Lemma~\ref{lem:e1-4})}\displaybreak[0]\\
    &=\sum_{j=0}^s (-1)^j F_{j}^{(s)} (n-1)_{s-j}\left(\frac{1}{q}\lambda_{s-j}-\lambda_{s-j}\right)\displaybreak[0]\\
    &=\left(\frac{1}{q}-1\right)\sum_{j=0}^s (-1)^j F_{j}^{(s)} (n-1)_{s-j}\lambda_{s-j}.
    \end{align*}
    \end{small}
    Since $\frac{1}{q}-1\neq 0$, 
    it follows that $\sum_{j=0}^s (-1)^j F_{j}^{(s)} (n-1)_{s-j}\lambda_{s-j}=0$. 
    This shows Eq. \eqref{eq:e3}
    for $\ell=1$. The cases for $\ell\in\{2,\ldots,s-1\}$ are similarly proven. 
    \end{proof}

Lemma~\ref{lem:e11}(2) with $\ell\in\{1,\ldots,s-1\}$ yields a homogeneous system of linear equations whose unknowns are $F_j^{(s)}$ with $0\leq j\leq s$. By rewriting $s-j$ with $j$, we have the following: 
\begin{align*}
A {\bm f}={\bm 0}, 
\end{align*}
where 
$
    A=(
    (-1)^{s-j}(n-\ell)_{j}\lambda_{j}
    )_{\substack{\ell=1,\ldots,s-1\\ j=0,\ldots,s}}
    ,{\bm f}=(F_{s-\ell}^{(s)})_{\ell=0,\ldots,s}
$. 

\begin{lemma}\label{lem:e1}
    The reduced row echelon form of 
    the matrix $A$ is $\begin{bmatrix}I_{s-1} & {\bm c}_{s} & {\bm c}_{s+1} \end{bmatrix}$, where 
    the $i$-th entry of the column ${\bm c}_s$ is $-\frac{\binom{s-1}{i-1}(n-i)_{s-i}}{q^{s-i}}$ and the $i$-th entry of ${\bm c}_{s+1}$ is $\frac{((s-1)\binom{s-1}{i-1}-\binom{s-1}{i-2})(n-i)_{s-i+1}}{q^{s-i+1}}$ for $i\in\{1,\ldots,s-1\}$.
\end{lemma}
\begin{proof} 
By direct calculations.
    %It follows from the row of elementary operations. 
\end{proof}
By Lemma~\ref{lem:e1}, $F_j^{(s)}$ for $j\in\{2,\ldots,s\}$ and $F_1^{(s)}$ are related as follows:  
%\begin{lemma}\label{lem:e2}
\begin{small}
\begin{align}
F_j^{(s)}&= 
\frac{(n-s+j-1)_{j-1}}{q^{j-1}}\left(\binom{s-1}{j-1}F_1^{(s)}-\frac{((s-1)\binom{s-1}{j-1}-\binom{s-1}{j})(n-s)}{q}\right).\label{eq:e5}   
\end{align}
\end{small}
%\end{lemma}
Note that Eq. \eqref{eq:e5} is valid for $j=1$. 

\begin{proof}[Proof of Theorem~\ref{thm:e2} for the lower bound]
By Eq. \eqref{eq:e5}, %Lemma~\ref{lem:e2}, 
$F_s^{(s)}$ and $F_1^{(s)}$ are related as follows:  
\begin{align*}
F_s^{(s)}= 
\frac{(n-1)_{s-1}}{q^{s-1}}\left(F_1^{(s)}-\frac{(s-1)(n-s)}{q}\right).   
\end{align*}
Since $F_s^{(s)}=\prod_{i=1}^s x_i\geq 0$, the above equality shows that 
\[
\sum_{i=1}^sx_i-\frac{s(s-1)}{2}=F_1^{(s)}\geq \frac{(s-1)(n-s)}{q}.
\]
This proves the lower bound on  $\sum_{i=1}^s x_i$. Equality occurs if and only if $F_s^{(s)}=\prod_{i=1}^s x_i=0$.  Since $0\leq x_1<\cdots<x_s$, the latter is equivalent to $x_1=0$.  
\end{proof}

Next, we turn to the upper bound on $\sum_{i=1}^s x_i$. 
By Lemma~\ref{lem:e1-4} and Eq. \eqref{eq:e2} with the fact that $\lambda_i=\frac{|C|}{q^i}$, we have: 
\begin{align}
    \sum_{j=0}^s (-1)^{j}F_{j}^{(s)}(n)_{s-j}=|C|\sum_{j=0}^s (-1)^j F_{j}^{(s)}\frac{(n)_{s-j}}{q^{s-j}}. \label{eq:e6}
\end{align}
Since $C$ is an $s$-distance set, $|C|\leq \sum_{k=0}^s \binom{n}{k}(q-1)^k=:M$ holds with equality if and only if $C$ is a tight $2s$-design.  
Since both sides of Eq. \eqref{eq:e6} are positive, it yields 
\begin{align}
    \sum_{j=0}^s (-1)^{j}F_{j}^{(s)}(n)_{s-j}\leq M\sum_{j=0}^s (-1)^j F_{j}^{(s)} \frac{(n)_{s-j}}{q^{s-j}}. \label{eq:e61}
\end{align}

Substituting Eq. \eqref{eq:e5} into Eq. \eqref{eq:e61} and simplifying it with $(n)_{s-j}(n-s+j-1)_{j-1}=\frac{(n)_{s}}{(n-s+j)}$ and $(n)_s>0$ gives
     \begin{align}
   &\sum_{j=1}^s \left(\frac{M}{q^{s-j}}-1\right) \frac{(-1)^{j+1}}{(n-s+j)q^{j-1}}\binom{s-1}{j-1}F_1^{(s)} \nonumber\\
   &\leq \sum_{j=0}^s \left(\frac{M}{q^{s-j}}-1\right) \frac{(-1)^{j+1}(n-s)}{(n-s+j)q^{j}}\left((s-1)\binom{s-1}{j-1}-\binom{s-1}{j}\right).
   \label{eq:e9}
    \end{align}
Next, we simplify the inequality in Eq. \eqref{eq:e9}.
\begin{lemma}\label{lem:esum}
Let $M=\sum_{k=0}^s \binom{n}{k}(q-1)^k$. Then:
\begin{enumerate}
    \item $\sum_{j=1}^s \left(\frac{M}{q^{s-j}}-1\right) \frac{(-1)^{j+1}}{(n-s+j)q^{j-1}}\binom{s-1}{j-1}=\frac{(q-1)^s}{sq^{s-1}}$, 
    \item $\sum_{j=0}^s \left(\frac{M}{q^{s-j}}-1\right) \frac{(-1)^{j+1}(n-s)}{(n-s+j)q^{j}}\left((s-1)\binom{s-1}{j-1}-\binom{s-1}{j}\right)=\frac{(n-s)(q-1)^s}{q^{s}}$. 
\end{enumerate}
\end{lemma}
\begin{proof}
   (1) is equivalent to 
\begin{align}
    \sum_{k=0}^s \left(\sum_{j=0}^{s-1}  \left(\frac{\binom{n}{k}-\binom{j}{k}}{n-j}\right) {(-1)^{(s-j)+1}}\binom{s-1}{j}\right)(q-1)^{k}=\frac{(q-1)^s}{s}.  \label{eq:econj1}
\end{align}
Note that $\frac{\binom{n}{k}-\binom{j}{k}}{n-j}$ is a polynomial in $j$ provided that $k<n$. 
Then using the binomial identity 
for the Stirling numbers of the second kind
\[
%\sum_{j=0}^m (-1)^j \binom{m}{j}j^k=\begin{cases}
%0 & \text{ if }0\leq k<m,\\
%m! & \text{ if }k=m, 
%\end{cases}
\sum_{j=0}^m (-1)^j \binom{m}{j}j^k=0
\]
for $0\leq k<m$, 
we get that the inner sum 
on the left-hand side of
Eq. \eqref{eq:econj1} is 
equal to $0$ if $k<s$.
%(it was: the term concerning with respect to $j$ appearing on the left-hand side of \eqref{eq:econj1} is equal to $0$ if $k<s$.) 

Recall Melzak's formula \cite{M51} which states that, for a polynomial $f(x)$ of degree $n$, we have  
\[
\sum_{k=0}^n (-1)^k \binom{n}{k}\frac{f(y-k)}{x+k}=\frac{f(x+y)}{x\binom{x+n}{n}}
\]
where $x\neq -k$ ($k=0,1,\ldots,n$). 
Thus, 
the inner sum for $k=s$ in Eq. \eqref{eq:econj1} is 
\begin{align*}
\sum_{j=0}^{s-1}  \left(\frac{\binom{n}{s}-\binom{j}{s}}{n-j}\right) {(-1)^{(s-j)+1}}\binom{s-1}{j}&=(-1)^{s+1}\binom{n}{s}\sum_{j=0}^{s-1}  \frac{1}{n-j} {(-1)^{j}}\binom{s-1}{j}\\
&=(-1)^{s+1}\binom{n}{s}\cdot (-1)^{s+1}  \frac{1}{n\binom{n-1}{s-1}}\\
%&=\frac{(-1)^{s+1}\binom{n}{s}}{n\binom{s-1-n}{s-1}}\\
&=\frac{1}{s}, 
\end{align*}
which proves part (1). Part (2) is similarly proven. 
\end{proof}
\begin{proof}[Proof of Theorem~\ref{thm:e2} for the upper bound] 
By Lemma~\ref{lem:esum}, 
the inequality in Eq. \eqref{eq:e9} simplifies to 
%\[
%F_1^{(s)}\leq \frac{\frac{(n-s)(q-1)^s}{q^{s}}}{\frac{(q-1)^s}{sq^{s-1}}}=\frac{(n-s)s}{q},  
%\]
%that is
\[
\sum_{i=1}^s x_i\leq \frac{(n-s)s}{q}+\frac{s(s-1)}{2}. 
\]
This proves the upper bound on 
$\sum_{i=1}^s x_i$, with  
equality if and only 
if $|C|=M$, i.e., $C$ is 
a tight $2s$-design.  
\end{proof}

%\subsection{Prop 4.1 in \cite{IS}}
We give without proof the following analogue of \cite[Proposition~4.1]{IS}.  
\begin{theorem}\label{thm:e3}
Let $C$ be a $(2s-1)$-design in $H(n,q)$ with degree $s$. %with index $\lambda_{2s-1}$, where $\lambda_{2s-1}=\frac{|C|}{q^{2s-1}}$. 
%Regarding $C$ as an $i$-design, its index $\lambda_i$ is $\lambda_i=\frac{|C|}{q^i}$. 
Define %$S(C)=\{\partial(x,y)\mid x,y\in C,x \neq y\}$ and 
$S'(C)=\{n-\partial(x,y)\mid x,y\in C,x \neq y\}$ and set $S'(C)=\{x_1,\ldots,x_s\}$ with $x_1<\cdots<x_s$. 
Then for all $m\in\{1,\ldots,s\}$,
\begin{align*}
\sum_{i=1}^m x_i\geq m(m-1),     
\end{align*}
with equality if and only if $x_i=2i-2$ for $i\in\{1,\ldots,m\}$.  
\end{theorem}
\begin{comment}
\begin{proof}
    We prove the claim by induction on $s$. 
    For $s=1$, it is obvious. 
    Let $s\geq 2$ and assume that the claim holds for $(2t-1)$-design with $t<s$. 

    The case $m=1$ is obvious, so we may assume that $m\geq 2$. 
    If $x_i\geq 2i$ for $i\in\{1,\ldots,m-1\}$, the claim holds. 
    Suppose $x_i\leq 2i-1$ for some $i\in\{1,\ldots,m-1\}$. 
    Let $\ell$ be the first index such that $x_\ell\leq 2\ell-1$

    Note that $n>2\ell$. Consider the design $C'$ obtained by iterating contraction $2\ell$ times; 
    \[
    C'=\{(c_{2\ell+1},\ldots,c_n)\mid (1,\ldots,1,c_{2\ell+1},\ldots,c_n)\in C\}.
    \]
    Then $C'$ is a $(2s-1-2\ell)$-design in $H(n-2\ell,q)$ with $S'(C')\subset \{x_1-2\ell,\ldots,x_s-2\ell\}$. 
    Since $x_i<2\ell-1$ for $i<\ell$, $S'(C')$ is a subset of $\{x_{\ell+1}-2\ell,\ldots,x_s-2\ell\}$. 
    Therefore, $C'$ has degree at most $s-\ell$. 

    By the induction hypothesis, $\sum_{i=\ell+1}^s (x_i-2\ell)\geq (s-\ell)(s-\ell-1)$, that is, 
    \[
    \sum_{i=\ell+1}^s x_i\geq s(s-1)-\ell(\ell-1)  
    \]
    with equality if and only if $x_i=2i-2$ for $i\in\{\ell+1,\ldots,s\}$. 
    Since $x_i\geq 2i$ for $i\in\{1,\ldots,\ell-1\}$
    \[
    \sum_{i=1}^\ell x_i\geq \ell(\ell-1)+x_\ell.
    \]
    Thus, 
    \[
    \sum_{i=1}^s x_i\geq (\ell(\ell-1)+x_\ell)+(s(s-1)-\ell(\ell-1))=s(s-1)+x_\ell\geq s(s-1). 
    \]
    Equality holds if and only if $\ell=1$ and $x_i=2i-2$ for any $i$. 
\end{proof}
\end{comment}
\begin{remark}
It is tempting to conjecture 
a stronger inequality $\sum_{i=1}^m x_i\geq \frac{qm(m-1)}{2}$. 
However, a counterexample to this is given 
by tight $3$-designs in $H(n,q)$, $q\geq 3$. 
%by Example~\ref{ex:1} (4). 
%\textcolor{red}{what was in that example? We seemed to delete it.}
\end{remark}

\subsection{Application: feasible parameters and examples}\label{sect:f}
A tuple of parameters $(|C|,n,q,\alpha_1,\ldots,\alpha_s)$ 
of a (putative) $t$-design $C$ in $H(n,q)$ 
with $t\geq 2s-2$ determines  
the second eigenmatrix $Q$ 
of the Delsarte scheme 
$(C,\{S_i\}_{i=0}^s)$ 
(see Theorem \ref{thm:t2s-2}), 
whence one can compute the intersection numbers and  
the Krein parameters of the scheme. 
We call such a tuple \emph{feasible} if the intersection numbers 
are nonnegative integers and the Krein parameters are 
nonnegative.

We list some feasible tuples of parameters for 
%We list some feasible tuples of parameters for \emph{nontight}
%\textcolor{red}{\bf I think, we also get tight examples in the %tables?}
$(2s-1)$-designs 
with degree $s$ in $H(n,q)$ for $s\in\{2,3,4\}$ in Tables~\ref{table:3des}, \ref{table:5des}, \ref{table:7des}, 
respectively. 
To find these parameters, we carried out the following procedure:
\begin{itemize}
    \item for each $q$ and $n$ in the given ranges, 
    compute the lower $L$ and upper $U$ bounds for $|C|$ 
    by Eq. \eqref{eq:tight} and Eq. \eqref{eq:tightcode}, %$\sum_{i=0}^s{n\choose i}(q-1)^i$, 
    respectively;
    \item for each value of $|C|$, $L\leq |C|\leq U$, divisible 
    by $q^{2s-1}$, compute the first eigenmatrix $P=|C|Q^{-1}$, 
    where $Q$ is defined in Lemma \ref{thm:t2s-2}, 
    with $\alpha_1,\ldots,\alpha_s$ being treated as indeterminates;
    \item the first row of $P$ gives the valencies of 
    the Delsarte scheme, whence one can compute the numbers of pairs $(x,y)\in C\times C$ such 
    that $\partial(x,y)=\alpha_i$, for $i=1,2,\ldots,s$;
    \item with this information in hand, simplify the left-hand 
    side of the equation in Lemma \ref{lem:cha}(3) to get a polynomial 
    $p_i$ in $\alpha_1,\ldots,\alpha_s$, for $i=1,2,\ldots,t$;
    \item compute a Gr\"{o}bner basis of the ideal 
    in $\mathbb{Q}[\alpha_1,\ldots,\alpha_s]$ generated 
    by the polynomials $p_i$, $i=1,2,\ldots,t$;
    \item by the elimination property, if the set of solutions 
    to the equations in Lemma \ref{lem:cha}(3) is finite, 
    the Gr\"{o}bner basis contains exactly one univariate 
    polynomial (say, in $\alpha_s$), whose roots are 
    the values of $\alpha_1,\ldots,\alpha_s$ (note that 
    each $\alpha_i$ should satisfy $\alpha_i\in \{0,\ldots,n\}$). Once these values are determined, we can check the feasibility conditions. 
    The case when the equations have infinitely many solutions 
    can be handled separately; in fact:
    \begin{itemize}
        \item for $s\in \{2,3\}$, it did not occur in our computations at all;
        \item for $s=4$, in all cases that occured, the largest polynomial of the basis factorized as $g(\alpha_4)\cdot \prod_{i<j}(\alpha_i-\alpha_j)$, 
        where $g$ is a univariate polynomial.
    \end{itemize}
\end{itemize}

  {\centering
  {\small 
  \begin{tabular}{ccccccc}
    \hline
    $|C|$ & $n$ & $q$ & $\alpha_1$ & $\alpha_2$ & $(v,k,\lambda,\mu)$ & Comment, see \cite{aeb} \\
    \hline \hline
     16 & 5 & 2 & 2 & 4 & $(16,5,0,2)$ & $\exists!$, \cite[Ex.~TF3]{CK}\\
    392 & 46 & 2 & 21 & 28 & $(392,115,18,40)$ & $\exists?$\\
1080 & 78 & 2 & 36 & 45 & $(1080,364,88,140)$ & $\exists?$\\
800 & 85 & 2 & 40 & 50 & $(800,204,28,60)$ & $\exists?$\\
784 & 116 & 2 & 56 & 70 & $(784,116,0,20)$ & $\exists?$\\
1600 & 205 & 2 & 100 & 120 & $(1600,205,0,30)$ & \\
8400 & 222 & 2 & 105 & 120 & $(8400,3367,1190,1456)$ & \\
4032 & 261 & 2 & 126 & 144 & $(4032,1015,182,280)$ & \\
13872 & 286 & 2 & 136 & 153 & $(13872,5720,2128,2520)$ &\\
7776 & 300 & 2 & 144 & 162 & $(7776,2600,736,936)$ &\\
81 & 10 & 3 & 6 & 9 & $(81,20,1,6)$ & $\exists!$, \cite[Ex.~TF3]{CK}\\
243 & 11 & 3 & 6 & 9 & $(243,110,37,60)$ & $\exists$, \cite[Ex.~RT6]{CK}\\
729 & 56 & 3 & 36 & 45 & $(729,112,1,20)$ & $\exists!$,  \cite[Ex.~FE2]{CK}\\
7803 & 235 & 3 & 153 & 170 & $(7803,1692,261,396)$ &\\
64 & 6 & 4 & 4 & 6 & $(64,18,2,6)$ & $\exists$, Example \ref{ex:1}(3) \\
256 & 17 & 4 & 12 & 16 & $(256,51,2,12)$ & $\exists$, \cite[Ex.~TF3]{CK} \\
4096 & 78 & 4 & 56 & 64 & $(4096,1287,326,440)$ & 
$\exists$, \cite[Ex.~FE3]{CK}
\\
625 & 26 & 5 & 20 & 25 & 
$(625,104,3,20)$ & $\exists$, \cite[Ex.~TF3]{CK} \\
28125 & 267 & 5 & 210 & 225 & $(28125,6764,1363,1710)$ &\\
216 & 8 & 6 & 6 & 8 & $(216,75,18,30)$ & $\not\exists$ Corollary~\ref{cor:tight3}\\
1296 & 37 & 6 & 30 & 36 & $(1296,185,4,30)$ & $\exists?$\\
    \hline
  \end{tabular}
\captionof{table}{Feasible parameters for non-tight $3$-designs with degree $2$ in $H(n,q)$ for $n\leq 300$ and $2\leq q \leq 6$.}
  \label{table:3des}}}

{\centering
  {\small 
  \begin{tabular}{ccccccc}
    \hline
    $|C|$ & $n$ & $q$ & $\alpha_1$ & $\alpha_2$ & $\alpha_3$ & Comment  \\
    \hline \hline
     32 & 6 & 2  & 2 & 4 & 6 &  the dual of the repetition  code\\
%     $(7, 2, 64, 2, 4, 6)$ \\
     1024 & 22 & 2 & 8 & 12 & 16 & the doubly shortened code of the extended Golay code\\ %in $\mathbb{F}_2^{23}$\\
%     $(23, 2, 2048, 8, 12, 16)$ \\
     729 & 12 & 3 & 6 & 9 & 12 & the dual code of the extended Golay code\\% in $ \mathbb{F}_3^{12}$\\
    \hline
  \end{tabular}
  \captionof{table}{Feasible parameters for $5$-designs with degree $3$ in $H(n,q)$, $2\leq n\leq 200,2\leq q\leq 10$.}
  \label{table:5des}  }}

  {\centering
  {\small 
  \begin{tabular}{cccccccc}
    \hline
    $|C|$ & $n$ & $q$ & $\alpha_1$ & $\alpha_2$ & $\alpha_3$ & $\alpha_4$ & Comment   \\
    \hline \hline
         128 & 8 & 2  & 2 & 4 & 6 & 8 & the dual of the repetition  code\\
     4096 & 24 & 2  & 8 & 12 & 16 & 24 & the extended Golay code, oa.4096.12.2.7 in \cite{Sl}\\
    \hline
  \end{tabular}
  \captionof{table}{Feasible parameters for $7$-designs with degree $4$ in $H(n,q)$, $8\leq n\leq 100,2\leq q\leq 10$.}
  \label{table:7des}  }}

As for Table~\ref{table:3des}, it is clear that  
the Delsarte scheme with $s=2$ classes corresponds to a strongly regular graph; the parameters are given in the second to last column. 

Recall that a linear code $C\leq \mathbb{F}_q^n$ is said to be \textbf{projective} if any two of its coordinates are linearly independent, i.e., the dual code $C^{\perp}$ has minimum distance $d_{C^\perp}\geq 3$. If $\dim C=k$, then the columns of the generator matrix of $C$ give a set $\mathcal{O}$ of $n$ points in the projective geometry $\mathrm{PG}(k-1,q)$. A code $C$ is called a \textbf{two-weight} code if every nonzero vector from $C$ has weight (support) $w_1$ or $w_2$. It is well known that the matrix whose rows consist of the vectors of $C$ is an $\OA(|C|,n,q,d_{C^\perp}-1)$. Thus, a projective two-weight code in $\mathbb{F}_q^{n}$ gives rise to a 2-design with degree $2$ in $H(n,q)$ (see \cite{CK} and \cite[Chapter~7]{BM} for more results on projective two-weight codes and strongly regular graphs). Furthermore, 
it has strength $3$ if and only if $d_{C^\perp}\geq 4$ holds if and only if no 3 points of $\mathcal{O}$ are on a line, i.e., $\mathcal{O}$ is a projective $n$-cap.
Projective $n$-caps were classified in a series of papers by Calderbank, Tzanakis and Wolfskill and others (see \cite[Section~7.1.9.J]{BM}):
\begin{itemize}
    \item[(i)] an ovoid of $\mathrm{PG}(3,q)$, \cite[Ex.~TF3]{CK}; 
    \item[(ii)] the Coxeter $11$-cap of $\mathrm{PG}(4,3)$, \cite[Ex.~RT6]{CK}; 
    \item[(iii)] the Hill $56$-cap of $\mathrm{PG}(5,3)$, \cite[Ex.~FE2]{CK}; 
    \item[(iv)] a $78$-cap of $\mathrm{PG}(5,4)$, \cite[Ex.~FE3]{CK}; 
    \item[(v)] a $430$-cap of $\mathrm{PG}(6,4)$.
\end{itemize}

The uniqueness of the example in (iv) and the existence of (v) remain longstanding open problems in finite geometry and coding theory, unresolved for more than 40 years. Recently, Bamberg \cite{Bamberg} showed that the answer to the latter question is negative if the answer to the former one is affirmative; the proof is based on a certain $9$-class association scheme. In fact, the Delsarte scheme of the design corresponding to a projective two-weight code is a translation scheme and thus admits the dual scheme (see \cite[Section~2.10.B]{BCN}). The Bamberg scheme is a fission of this dual scheme for the only known example in (iv). Interestingly, it admits a 4-class fusion with the same parameters 
as the fission scheme of the Delsarte scheme of the design from Theorem \ref{thm:qant4}; we did not check whether these two schemes are isomorphic.
%does there exist a two-character $430$-cap of $\mathrm{PG}(6,4)$ (i.e., a $3$-design for $\alpha_1=320$ and $\alpha_2=352$, $n=430$ and $q=4$)? 

%\textcolor{red}{We need to show that if the dual code $C^{\perp}$ has min distance at least 4, then a projective two-weight code, i.e., a \emph{cap},  gives an OA of strength 3, see \cite[7.1.9.J]{BM}. It seems that all examples from Table 7.1 in \cite{BM} give examples in our Table 1. I think the key property is that a cap does not have three collinear points, this should make an OA of strength 3. Also this explains why we get these codes (and not their duals) in our Table.}

%\textcolor{red}{$C:$ a $q$-ary two-weight $[n,m]$-code with weights $w_1,w_2$. Then every hyperplane meets $X$ in $n-w_i$ points \cite[Section~7.1.2]{BM}. ... a 430-cap of PG(6, 4) is such that every hyperplane intersects in 78 or 110 elements. It corresponds to a linear code with weights $430-110=320$ and $430-78=352$.}

%\textcolor{red}{$OA(n,q,\alpha_1,\alpha_2)\longleftrightarrow [n,m]_q$ with weights $\alpha_1,\alpha_2$}. 

%Using Lemma \ref{lem:cha}, one can derive an equation (which we omit here) for $n,q$ and the weights of a code to guarantee that this design has strength $3$. 
Since a projective two-weight 
code is linear, the number 
of points must be divisible 
by a power of $q$. This shows 
that in many open cases 
from Table \ref{table:3des} such 
designs cannot be obtained from codes\footnote{For  $n=8$, $q=6$, note that a strongly regular graph may still exist, see \cite{aeb}.}. 
We wonder whether  
the fission schemes of their 
Delsarte schemes may help 
to rule out their existence.

\begin{comment}
\textcolor{red}{(1) Need to identify which of them are realized as $C$? 
(2) A projective two-weight code gives an OA 
of strength 2. Parameters of projective two-weight codes which give extremal OA 
of strength 3. (3) Mention that the fission 
scheme could be used to show non-existence?
}
\textcolor{blue}{
$\rightarrow$(1) Proejcctive two-weight codes are supposed to be linear, so the number of codewords in $C$ must be a prime power. Many in Table~1 are excluded. (2) For a $2$-design $C$ with degree set $\{d_1,d_2\}$ in $H(n,q)$, the parameter must satisfy that 
\[
d_1 q (d_2 (q-q v)+n (q-1) v)-n (q-1) v (-d_2q+n (q-1)+1)=0,
\]
which comes from $2$-design condition. 
$C$ is a $3$-design iff 
\begin{align*}
&d_1^2 q^2 (d_2 q (v-1)-n (q-1) v)+d_1 q ((d_2+3) q-3 n (q-1)-6) (d_2 q (v-1)-n (q-1) v)\\
&-n (q-1) v \left(-3 n (q-1) (d_2 q+q-2)+d_2 (d_2+3) q^2-2 (3 d_2+1) q+2 n^2 (q-1)^2+4\right)=0.
\end{align*}
(We can solve these equations with respect to $d_1,d_2$. But it isn't very easy, so I don't include this here. However, I use this closed formula to make Table feasible parameters. )
(3) I had thought of fission schemes of classes $3$, but not obtained except for tight $3$-designs. }

In Table~\ref{table:5des}, the first row corresponds to the dual of the repetition binary code of length $6$.
\end{comment}

The degree sets of some examples 
in Tables \ref{table:3des}, 
\ref{table:5des}, 
\ref{table:7des} feature 
a symmetry property. 
We will characterize 
these examples in the 
theorems below. 
%The 7-designs from Table~\ref{table:7des} 
%are the dual of the binary repetition code of length 8 and 
%the extended Golay code of length 24, respectively.
%In the next theorems, we characterize these codes as designs 
%with certain degree sets.
More precisely, we consider  codes $C$ in $H(n,2)$ having 
the property that both $n\in S(C)$ and $n-a\in S(C)$ 
whenever $ n\ne a\in S(C)$. 
These codes include the class of {\it self complementary codes}, 
see \cite{AHS17} for linear programming bounds 
and related association schemes. 
Before stating the results, we prepare a lemma for extremal designs $C$ with the degree set $S(C)$ such that $a\in S(C)\cup\{0\}$ implies $n-a\in S(C)\cup\{0\}$. 
\begin{lemma}\label{lem:q2}
    Let $C$ be a $(2s-1)$-design with degree $s$ in $H(n,q)$. If, for every $a\in S(C)\cup\{0\}$, one has $n-a\in S(C)\cup\{0\}$, then $q=2$.  
\end{lemma}
\begin{proof}
    As in Theorem \ref{thm:e2}, 
    define the set $S'(C)=\{x_1,\ldots,x_s\}$ with $x_1<\cdots<x_s$. By our assumptions, $x_1=0$ and 
    the set $S(C)\cup\{0\}$ is symmetric with respect to $n/2$, hence Theorem~\ref{thm:e2} shows that 
    \[
    \frac{(s-1)n}{2}=\sum_{i=1}^s x_i=\frac{(s-1)(n-s)}{q}+\frac{s(s-1)}{2}, 
    \]
    whence it follows that $q=2$. 
\end{proof}
\begin{theorem}
Let $n,q$ be positive integers such that $n$ is even and $n\geq 4,q\geq2$.   
Any $3$-design $C$ in $H(n,q)$ with degree $2$ and degree set $S(C)=\{n/2,n\}$ is isomorphic to the Hadamard code of a Hadamard matrix of order $n$. 
\end{theorem}
\begin{proof}
By Lemma~\ref{lem:q2}, $q=2$. 
By Theorem~\ref{thm:t2s-2}, 
the Delsarte scheme of $C$ 
%a $3$-design $C$ has a structure of 
%a $Q$-polynomial association scheme with
has the first eigenmatrix: 
\[
\left[
\begin{array}{ccc}
 1 & |C|-2 & 1 \\
 1 & \frac{|C|}{n}-2 & 1-\frac{|C|}{n} \\
 1 & -2 & 1 \\
\end{array}
\right].
\]
By Lemma~\ref{lem:cha}(3), the only possible value for $|C|$ is $|C|=2n$. 
It is easily seen that $C$ is the Hadamard code of a Hadamard matrix 
of order $n$, and vice versa. 
\end{proof}

\begin{theorem}
Let $n,q$ be positive integers such that $n\geq 6,q\geq2$. 
Any $5$-design $C$ in $H(n,q)$ with degree $3$ and degree set $S(C)=\{a,n-a,n\}$ for some $a$ with $a<n/2$ is isomorphic to the dual of the binary repetition code of length $6$. 
\end{theorem}
\begin{proof}
By Lemma~\ref{lem:q2}, $q=2$. 
By Theorem~\ref{thm:t2s-2}, 
the Delsarte scheme of $C$ has 
%a $5$-design $C$ has a structure of a $Q$-polynomial association scheme with 
the first eigenmatrix: 
\[
\left[
\begin{array}{cccc}
 1 & \frac{4n (a-n)+n |C| (-2a +n+1)}{4 \left(2 a^2-3 a n+n^2\right)} & \frac{n (2a (|C|-2)+|C| (-2n+n+1))}{4a (2 a-n)} & |C|-1+\frac{(n-1) n |C|}{4a (a-n)} \\
 1 & \frac{2n (a-n)+|C| (n-a)}{2 \left(2 a^2-3 a n+n^2\right)} & \frac{a (|C|-2n)}{2a (2 a-n)} & -1 \\
 1 & \frac{2n(a-n)+|C|}{2\left(2 a^2-3 a n+n^2\right)} & \frac{|C|-2a n}{2a(2 a-n)} & -\frac{|C|}{2a(a-n)}-1 \\
 1 & \frac{n}{2 a-n} & \frac{2 a}{n-2 a}+1 & -1 \\
\end{array}
\right].
\]
By Lemma~\ref{lem:cha}(3), the only possible values for $|C|$ and $a$ are $(|C|,a)=(n^2-n+2,(n-\sqrt{n-2})/2)$. 
Thus, $n=m^2+2$ for a positive integer $m$ and 
the first eigenmatrix becomes   
\[\left[
\begin{array}{cccc}
 1 & \frac{1}{2} (m^2+1) (m^2+2) & \frac{1}{2}(m^2+1)(m^2+2) & 1 \\
 1 & -\frac{1}{2} m(m^2+1) & \frac{1}{2} m(m^2+1) & -1 \\
 1 & -1 & -1 & 1 \\
 1 & \frac{m^2+2}{m} & -\frac{m^2+2}{m} & -1 \\
\end{array}
\right].\]
Since the (4,1)-entry must be an integer, $m=1$ or $2$. 
Then $(n,|C|,a)=(3,8,1)$ or $ (6,32,2)$. 
The first case $n=3$ cannot afford the assumption $n\geq 6$. 
For the second case $n=6$, it is easy to see that $C$ is 
the dual of the binary repetition code.    
\end{proof}

\begin{theorem}
Let $n,q$ be positive integers such that $n$ is even and $n\geq 8,q\geq2$. 
Any $7$-design $C$ in $H(n,q)$ with degree $4$ and degree set $S(C)=\{a,n/2,n-a,n\}$ for some $a$ with $a<n/2$ is isomorphic to either the dual of the binary repetition code of length $8$ or the extended Golay code of length $24$. 
\end{theorem}
\begin{proof}
As in the preceding theorems, 
$q=2$ and computing the first 
eigenmatrix of the Delsarte scheme of $C$
shows that $(|C|,a)=(n(n^2-3n+8)/3,(n-\sqrt{3 n-8})/2)$. 
Thus, $n=(m^2+8)/3$ for a positive integer $m$ and 
the first eigenmatrix becomes   
\[\left[
\begin{array}{ccccc}
 1 & \frac{n^2 \left(m^4+7 m^2+10\right)}{54 m^2} & \frac{2 \left(m^8+11 m^6+87 m^4+221 m^2-320\right)}{243 m^2} & \frac{n^2 \left(m^4+7 m^2+10\right)}{54 m^2} & 1 \\
 1 & \frac{n \left(m^4+7 m^2+10\right)}{54 m} & 0 & -\frac{m^6+15 m^4+66 m^2+80}{162 m} & -1 \\
 1 & \frac{n \left(m^4-2 m^2-8\right)}{27 m^2} & -\frac{2 \left(m^6+6 m^4+57 m^2-64\right)}{81 m^2} & \frac{m^6+6 m^4-24 m^2-64}{81 m^2} & 1 \\
 1 & -\frac{n}{m} & 0 & \frac{n}{m} & -1 \\
 1 & -\frac{n^2}{m^2} & \frac{2 \left(m^4+7 m^2+64\right)}{9 m^2} & -\frac{n^2}{m^2} & 1 \\
\end{array}
\right].\]
Since the (4,1)-entry must be an integer, $m\in\{1,2,4,8\}$. 
Then $(n,|C|,a)=(3,8,1),$ $(4,16,1),$ $(8,128,2),(24,4096,8)$. 
The first two cases $n\in\{3,4\}$ do not afford the assumption $n\geq 8$.  
For the third case $n=8$, it is easy to see that $C$ is the dual 
of the binary repetition code of length $8$.    
For the last case, $C$ is the extended Golay code of length $24$.  
\end{proof}

%\section{Examples}\label{sect:e}
\begin{example}\label{ex:1}    
%For the case $s=2$, projective two-weight codes whose dual code has minimum distance at least four are examples. 
Here we list all known 
%known examples 
%of $(2s-1)$-designs of degree $s$ (for $s=2,3,4$) and 
infinite 
families of $(2s-1)$-designs of degree $s$ in Hamming schemes:
\begin{enumerate}
\item the dual of the binary repetition code of length $2s$ is a tight $(2s-1)$-design in $H(2s,2)$ with degree $s$ for any $s$; 
\item a Hadamard code of length $4n$ is a tight $3$-design in $H(4n,2)$ with degree set $\{2n,4n\}$.
\item for any prime power $q=2^m$, there is a tight $3$-design in $H(q+2,q)$ with degree set $\{q,q+2\}$. 
%\textcolor{red}{why don't we mention here the tight one from Cor. 3.5?}
\item for any prime power $q$, there is a nontight $3$-design in $H(q^2+1,q)$ with degree set $\{q^2-q,q^2\}$, see \cite[Example~TF3]{CK}.

\end{enumerate}
\end{example}

\subsubsection*{Acknowledgments}
We are grateful to Jason Williford for pointing out the connection between the scheme from Theorem \ref{thm:main} and linked systems of symmetric designs.

Sho Suda's research is supported by JSPS KAKENHI Grant Number 22K03410.

\bibliographystyle{abbrv}
\bibliography{extremalOA}

\begin{thebibliography}{99}
\bibitem{AHS17}
M. Araya, M. Harada and S. Suda, 
Quasi-unbiased Hadamard matrices and weakly unbiased Hadamard matrices: a coding-theoretic approach, 
{\sl Math.\ of Comp}, {\bf 304} (2017), 951--984. 

\bibitem{BD1}
E. Bannai, R. M. Damerell, Tight spherical designs I,
{\sl J. Math. Soc. Japan} {\bf 31} (1980), 199--207.
\bibitem{BD2}
E. Bannai, R. M. Damerell, Tight spherical designs II,
{\sl J. London Math. Soc.} {\bf 21} (1980), 13--30.

\bibitem{BI}
E. Bannai and T. Ito, 
{\sl Algebraic Combinatorics I: Association Schemes},
{Benjamin/Cummings, Menlo Park, CA,} 1984.

\bibitem{BB}
R. C. Bose, K. A. Bush, 
Orthogonal arrays of strength two and three, 
{\sl Ann. Math. Stat.} {\bf 23}, (1952), 508--524.

\bibitem{aeb}
A.~E.~Brouwer, {\sl Parameters of Strongly Regular Graphs}, 
Available
at \url{https://aeb.win.tue.nl/graphs/srg/srgtab.html}

\bibitem{BCN}
A.~E.~Brouwer, A.~E.~Cohen, A.~Neumaier, 
{\sl Distance-regular graphs}, 
Springer-Verlag, Berlin, 1989. xviii+495.

\bibitem{BM}
A.~E.~Brouwer, H.~Van Maldeghem, 
{\sl Strongly Regular Graphs}, 
Encyclopedia of Mathematics and its Applications, Cambridge University Press, Cambridge, 2022. 

\bibitem{BNS}
P. Boyvalenkov, H. Nozaki, N. Safaei, 
Rationality of the inner products of spherical $s$-distance $t$-designs for $t\geq 2s−2$, $s\geq 3$,
{\sl Linear Algebra Appl.} {\bf 646} (2022), 107--118.

\bibitem{B}
K. A. Bush, 
Orthogonal arrays of index unity, 
{\sl Ann. Math. Statist.} {\bf 23} (1952), 426--434.

\bibitem{CG}
A. R. Calderbank, J.-M. Goethals, 
On a pair of dual subschemes of the Hamming scheme $H_n(q)$, 
{\sl European J. Combin.}, {\bf 6} (1985), 133--147.


\bibitem{CK}
A. R. Calderbank, W. M. Kantor, 
The geometry of two-weight codes, 
{\sl Bull. London Math. Soc.}, {\bf 18} (1986), 97--122.

\bibitem{CGS}
P. J. Cameron, J.-M. Goethals, J. J. Seidel,  Strongly regular graphs having strongly regular
subconstituents, 
{\sl J. Algebra}, {\bf 55} (1978), 257--280.

\bibitem{CJ}
K. Coolsaet, A. Juri\v{s}i\'{c}. Using equality in the Krein conditions to prove nonexistence of certain distance-regular graphs, 
{\sl J. Combin. Theory Ser. A}, {\bf 115} (2008), 1086--1095.


%\bibitem{DMM}
%E. van Dam, W. Martin, M. Muzychuk, 
%Uniformity in association schemes and coherent configurations: cometric $Q$-antipodal schemes and linked systems,
%{\sl J.\ Combin.\ Theory Ser.\ A} {\bf 120} (2013), 1401--1439. 

\bibitem{D}
P.~Delsarte, An algebraic approach to the association schemes of coding theory, 
Philips Res. Rep. 10 (Suppl.) (1973).

\bibitem{DGS}
P. Delsarte, J. M. Goethals, J. J. Seidel, Spherical codes and designs,
{\sl Geom. Dedicata} {\bf 6} (1977), 363--388.

\bibitem{DD}
K.~Ding, C. Ding,  
A class of two-weight and three-weight codes and their applications in secret sharing,  
{\sl IEEE}, {\bf 61} (2015),  5825--5842.

%\bibitem{GK} 
%A. L. Gavrilyuk, J. H. Koolen, 
%The Terwilliger polynomial of a $Q$-polynomial distance-regular graph and its application to pseudo-partition graphs, 
%{\sl Linear Algebra Appl.}, {\bf 466} (2015), 117--140.

\bibitem{GSV}
A. L. Gavrilyuk, S. Suda, J. Vidali,  
On tight $4$-designs in Hamming association schemes, 
{\sl Combinatorica}, {\bf 40} (2020), 345-362. 

\bibitem{GVW} A. L. Gavrilyuk, J. Vidali, J. S. Williford, On few-class $Q$-polynomial 
association schemes:
feasible parameters and nonexistence results, 
{\sl Ars Mathematica Contemporanea}, {\bf 20} (2021), 103-127. 

\bibitem{HSS}
A. S. Hedayat, N. J. A. Sloane, J. Stufken,  {\sl Orthogonal arrays, theory and applications}, New York: Springer, (1999). 

\bibitem{H}
Y. Hong, 
On the Nonexistence of Nontrivial perfect $e$-Codes and Tight $2e$-Designs in Hamming Schemes $H(n,q)$ with $e\geq 3$ and $q\geq 3$, 
{\sl Graphs and Combin.} {\bf 2} (1986), 145--164.

\bibitem{IS}
Y. Ionin, M. Shrikhande, $(2s-1)$ designs with $s$ intersection numbers,
{\sl Geom. Dedicata} {\bf 48} (1993), 247--265.
 
%\bibitem{JKT}
% A. Juri\v{s}i\'{c}, J. Koolen, P. Terwilliger,  Tight distance-regular graphs, 
% {\sl J. Algebraic Combin.}, {\bf 12} (2000), 163--197.
 
%\bibitem{JV2012}
% A. Juri\v{s}i\'{c}, J. Vidali, 
% Extremal 1-codes in distance-regular graphs of diameter 3, 
% {\sl Des. Codes Cryptogr.}, {\bf 65} (2012), 29–47. 
 
%\bibitem{JV2017} A. Juri\v{s}i\'{c}, J. Vidali. Restrictions on classical distance-regular graphs,  {\sl J. Algebraic Combin.}, {\bf 46} (2017) 571--588.
 
 
 \bibitem{KV}
 V. Kr\v{c}adinac, R. Vlahovi'{c} Kruc, 
Schematic $4$-designs, 
{\sl Discrete Math.} {\bf 346} (2023), 113385.
 
%\bibitem{MMW}
%W. J. Martin, M. Muzychuk, J. Williford, Imprimitive cometric association schemes: constructions and analysis, 
%{\sl J. Algebraic Combin.} {\bf 25} (2007), 399--415.

\bibitem{M51}
Z. A. Melzak, Problem 4458, {\sl Amer. Math. Monthly} {\bf 58} (1951), 636. 

\bibitem{MK94}
R. Mukerjee, S. Kageyama, 
On existence of two symbol complete orthogonal arrays, 
{\sl J.\ Combin.\ Theory Ser.\ A} {\bf 66} (1994), 176--181. 

\bibitem{N1979}
R. Noda, 
On orthogonal arrays of strength $4$ achieving Rao's bound,
{\sl J. London Math. Soc.} (1979),  385--390. 

\bibitem{N1986}
R. Noda, 
On orthogonal arrays of strength $3$ and $5$ achieving Rao's bound,
{\sl Graphs and Combin.} (1986), 277--282.

%\bibitem{RW}
%D. K. Ray-Chaudhuri, R. M. Wilson, 
%On $t$-designs, 
%{\sl Osaka J. of Math.\ }, {\bf 12}, (1975), 737--744.

\bibitem{R}
C.\ R.\ Rao, 
Factorial experiments derivable from combinatorial arrangements of arrays,
{\sl J. Roy. Statist. Soc.} {\bf 9} (1947), 128--139.

\bibitem{RW}
Dijen K. Ray-Chaudhuri, Richard M. Wilson,  
On $t$-designs, 
{\sl Osaka  J. Math.}, {\bf 12} (1975), 737--744.

\bibitem{Sl}
N.\ J.\ A.\ Sloane, A Library of Orthogonal Arrays, 
\url{http://neilsloane.com/oadir/}

\bibitem{S}
S. Suda, Coherent configurations and triply regular association schemes obtained from spherical designs. 
{\sl J. Combin. Theory Ser. A} {\bf 117} (2010), no. 8, 1178--1194.

%\bibitem{WR}
%R. M. Wilson, D. K. Ray-Chaudhuri, 
%Generalization of Fisher's inequality to $t$-designs, 
%{\sl Amer. Math.\ Soc.\ Notices}, {\bf 18}, (1971), 805.


\bibitem{S2022}
S. Suda, 
$Q$-polynomial coherent configurations, 
{\sl Linear Algebra its Appl.} {\bf 643} (2022), 166--195.

\bibitem{U}
 M. Urlep, 
 Triple intersection numbers of $Q$-polynomial distance-regular graphs, 
 {\sl European J. Combin.}, {\bf 33} (2012), 1246--1252.

\bibitem{JV}
J. Vidali, Using symbolic computation to prove nonexistence of distance-regular graphs. {\sl Electron. J. Combin.}, {\bf 25(4)} (2018), \#P4.21.


\end{thebibliography}

\begin{comment}

\end{comment}

\appendix
\def\thesection{Appendix \Alph{section}}

\begin{comment}
\section{Tight $3$-designs with $(|C|,n,q)=(2n,n,2)$}
We deal with the case $(|C|,n,q)=(2n,n,2)$ with $n$ a multiple of $4$. 
Then $\alpha_1=n/2,\alpha_2=n$. 
The fission scheme of $3$-class has the following second eigenmatrix $Q$ and the matrix $L_1^*$: 
\begin{align*}
Q=\left(
\begin{array}{cccc}
 1 & n-1 & n-1  & 1 \\
 1 & -1 & -1 & 1 \\
 1 & 1 & -1 & -1 \\
 1 & -n+1 & n-1 & -1 \\
\end{array}
\right),\quad
L_1^*=\left(
\begin{array}{cccc}
 0 & n-1 & 0 & 0 \\
 1 & 0 & n-2 & 0 \\
 0 & n-2& 0 & 1 \\
 0 & 0 & n-1 & 0 \\
\end{array}
\right).
\end{align*} 
This scheme coincides with one in \cite[Theorem~1.6.1]{BCN}. 
\end{comment}

\def\thesection{\Alph{section}}
%\section{The second eigenmatrix for nontight $3$-designs with degree $2$}
\section{Nontight $3$-designs with degree $2$}\label{app:A}

Let $C$ be a nontight $3$-design with degree $2$ in $H(n,q)$. 
Set $S(C)=\{\alpha_1,\alpha_2\}$, $\alpha_0=0$.  
Then the fission scheme of $4$ classes has the following second eigenmatrix: 
\begin{align}
\bbordermatrix{& E_{0,0}  & E_{0,1} & E_{0,2} & E_{1,0} & E_{1,1} \cr
A_{0,0} & 1 & K_{n-1,q,1}(\alpha_0) &  |C|-q\sum_{\ell=0}^1K_{n-1,q,\ell}(\alpha_0) & (q-1)K_{n-1,q,1}(\alpha_0) & q-1 \cr
A_{0,1} & 1 & K_{n-1,q,1}(\alpha_1) &  -q\sum_{\ell=0}^1K_{n-1,q,\ell}(\alpha_1) & (q-1)K_{n-1,q,1}(\alpha_1) & q-1 \cr
A_{0,2} & 1 & K_{n-1,q,1}(\alpha_2) &  -q\sum_{\ell=0}^1K_{n-1,q,\ell}(\alpha_2) & (q-1)K_{n-1,q,1}(\alpha_2) & q-1 \cr
A_{1,1} & 1 & K_{n-1,q,1}(\alpha_1-1) & 0 & -K_{n-1,q,1}(\alpha_1-1) & -1 \cr
A_{1,2} & 1 & K_{n-1,q,1}(\alpha_2-1) & 0 & -K_{n-1,q,1}(\alpha_2-1) & -1 }.\label{eq:fission4}
\end{align} 

There is a one-parameter family with $(|C|,n,\alpha_1,\alpha_2)=(q^4,q^2+1,q^2-q,q^2)$  \cite[Theorem~7.1.1]{BM}.
%There is a  one-parameter family with $(|C|,n,\alpha_1,\alpha_2)=(q^4,q^2+1,q^2-q,q^2)$, and it is constructed in \cite[Corollary~4]{DD} with $m=4$ by replacing $p$ in \cite[Corollary~4]{DD} with $q$. 
%\textcolor{red}{They seemed to be known from much earlier, see \cite[Theorem~7.1.1]{BM}.}

The eigenmatrices for the corresponding $4$-class scheme are 
\begin{align*}
P&=\left[
\begin{array}{ccccc}
 1 & q \left(q^2-1\right) & q-1 & (q-1)^2 q^2 & (q-1) q^2 \\
 1 & 0 & -1 & (q-1) q & -(q-1) q \\
 1 & -q & q-1 & 0 & 0 \\
 1 & 0 & -1 & -q & q \\
 1 & q (q^2-1) & q-1 & -(q-1) q^2 & -q^2 \\
\end{array}
\right],\\
Q&=\left[
\begin{array}{ccccc}
 1 & (q-1) q^2 & q (q^2-1) & (q-1)^2 q^2 & q-1 \\
 1 & 0 & -q & 0 & q-1 \\
 1 & -q^2 & q(q^2-1) & -(q-1) q^2 & q-1 \\
 1 & q & 0 & -q & -1 \\
 1 & -(q-1) q & 0 & (q-1) q & -1 \\
\end{array}
\right]. 
\end{align*} 

\section{Determinant of the second eigenmatrix}\label{sect:det}
Here we strengthen \cite[Theorem~1]{CG} by Calderbank and Goethals 
for $(2s-1)$-design with degree $s$. We adhere to the notation 
introduced in Section \ref{sect:fission}.

\begin{lemma}{\rm (See the proof of \cite[Theorem~1]{CG})}\label{lem:det}
Let $M=(K_{n,q,k}(\alpha_i))_{\substack{1\leq i \leq s\\ 0\leq k \leq s-1}}$. 
Then $$
\det M=\frac{q^{s(s-1)/2} \prod_{1\leq i<j \leq s}(\alpha_i-\alpha_j)}{\prod_{i=1}^{s-1}i!}. 
$$
\end{lemma}

Recall that the entries of the first eigenmatrix are algebraic integers.

\begin{lemma}\label{lem:Qint}
If the second eigenmatrix $Q$ of a $D$ classes symmetric association scheme with $v$ vertices has only integer entries, then $\det Q$ divides $v^{D+1}$. 
\end{lemma}
\begin{proof}
%Recall that $Q=\frac{1}{v}P^{-1}$ and the entries of the first eigenmatrix $P$ are algebraic integers. 
If the entries of the second eigenmatrix $Q=\frac{1}{v}P^{-1}$ are all integral, then 
those of the first eigenmatrix $P$ are all rational and hence must be integral as well. 
Taking the determinant of the equation $PQ=v I$, we obtain 
$\det P \det Q=v^{D+1}$. Therefore, $\det Q$ divides $v^{D+1}$, which yields the conclusion.
\end{proof}

As a corollary of the lemmas above, Calderbank and Goethals showed the following theorem. 
\begin{theorem}{\rm (\cite[Theorem~1]{CG})}\label{thm:CG}
If $C$ is a $(2s-2)$-design with degree $s$ 
and $S(C)=\{\alpha_1,\ldots,\alpha_s\}$ 
in $H(n,q)$,  
then $$\frac{q^{s(s-1)/2}\prod_{1\leq i<j \leq s}(\alpha_i-\alpha_j)}{\prod_{i=1}^{s-1}i!}$$ is an integer dividing $|C|^{s}$.
\end{theorem}
\begin{proof}
Follows from Lemma~\ref{lem:Qint}, Lemma~\ref{lem:det}, and Theorem~\ref{thm:t2s-2}.
\end{proof}

\begin{lemma}\label{lem:s}
Let $Q$ be the second eigenmatrix as in Theorem~$\ref{thm:qant4}(2)$. 
Then $$\det Q=\pm|C|q^{s} 
\left(\frac{q^{s(s-1)/2}\prod_{1\leq i<j \leq s}(\alpha_i-\alpha_j)}{\prod_{i=1}^{s-1}i!}\right)^2$$ 
and $(\det Q)/|C|$ is an integer. 
\end{lemma}
\begin{proof}
Compute the determinant of the second eigenmatrix $Q$ as follows.
Denote $K_{n-1,q,j}(x)$ by $K_{j}(x)$ for short. 
For $i=1,\ldots,s$, add the $i$-th column to the $(s+1+i)$-th column to obtain:
%Denoting by $K_{j}(x)=K_{n-1,q,j}(x)$ and $\det X=|X|$, we have  
\begin{align}
%&\begin{vmatrix}
% K_{0}(\alpha_0) &\cdots & K_{s-1}(\alpha_0) & |C|-q\sum_{\ell=0}^{s-1}K_{\ell}(\alpha_0) & (q-1)K_{0}(\alpha_0) &\cdots & (q-1)K_{s-1}(\alpha_0) \cr
% K_{0}(\alpha_1) &\cdots & K_{s-1}(\alpha_1)& -q\sum_{\ell=0}^{s-1}K_{\ell}(\alpha_1) & (q-1)K_{0}(\alpha_1) &\cdots & (q-1)K_{s-1}(\alpha_1) \cr
%  \vdots &\ddots & \vdots & \vdots & \vdots \cr
%   K_{0}(\alpha_s) &\cdots & K_{s-1}(\alpha_s)& -q\sum_{\ell=0}^{s-1}K_{\ell}(\alpha_s) & (q-1)K_{0}(\alpha_s) &\cdots & (q-1)K_{s-1}(\alpha_s) \cr
% K_{0}(\alpha_{1}-1) &\cdots & K_{s-1}(\alpha_{1}-1)& 0& -K_{0}(\alpha_{1}-1) & \cdots & -K_{s-1}(\alpha_{1}-1) \cr
% \vdots &\ddots & \vdots & \vdots& \vdots & \ddots & \vdots \cr
% K_{0}(\alpha_{s}-1) &\cdots & K_{s-1}(\alpha_{s}-1)& 0& -K_{0}(\alpha_{s}-1) & \cdots & -K_{s-1}(\alpha_{s}-1) 
%\end{vmatrix}\nonumber\displaybreak[0]\\
%&=
&\begin{vmatrix}
 K_{0}(\alpha_0) &\cdots & K_{s-1}(\alpha_0) & |C|-q\sum_{\ell=0}^{s-1}K_{\ell}(\alpha_0) & q K_{0}(\alpha_0) &\cdots & q K_{s-1}(\alpha_0) \cr
 K_{0}(\alpha_1) &\cdots & K_{s-1}(\alpha_1)& -q\sum_{\ell=0}^{s-1}K_{\ell}(\alpha_1) & q K_{0}(\alpha_1) &\cdots & q K_{s-1}(\alpha_1) \cr
  \vdots &\ddots & \vdots & \vdots & \vdots \cr
   K_{0}(\alpha_s) &\cdots & K_{s-1}(\alpha_s)& -q\sum_{\ell=0}^{s-1}K_{\ell}(\alpha_s) & q K_{0}(\alpha_s) &\cdots & q K_{s-1}(\alpha_s) \cr
 K_{0}(\alpha_{1}-1) &\cdots & K_{s-1}(\alpha_{1}-1)& 0& 0 & \cdots & 0 \cr
 \vdots &\ddots & \vdots & \vdots& \vdots & \ddots & \vdots \cr
 K_{0}(\alpha_{s}-1) &\cdots & K_{s-1}(\alpha_{s}-1)& 0& 0 & \cdots & 0 
\end{vmatrix}\nonumber\displaybreak[0]\\
&\text{then sum the last $s$ columns up and add them to the $(s+1)$-th column:}
% and factor $q$ out:}
\nonumber\displaybreak[0]\\
&=q^{s}\begin{vmatrix}
 K_{0}(\alpha_0) &\cdots & K_{s-1}(\alpha_0) & |C| & K_{0}(\alpha_0) &\cdots &  K_{s-1}(\alpha_0) \cr
 K_{0}(\alpha_1) &\cdots & K_{s-1}(\alpha_1)& 0 & K_{0}(\alpha_1) &\cdots &  K_{s-1}(\alpha_1) \cr
  \vdots &\ddots & \vdots & \vdots & \vdots \cr
   K_{0}(\alpha_s) &\cdots & K_{s-1}(\alpha_s)& 0 & K_{0}(\alpha_s) &\cdots & K_{s-1}(\alpha_s) \cr
 K_{0}(\alpha_{1}-1) &\cdots & K_{s-1}(\alpha_{1}-1)& 0& 0 & \cdots & 0 \cr
 \vdots &\ddots & \vdots & \vdots& \vdots & \ddots & \vdots \cr
 K_{0}(\alpha_{s}-1) &\cdots & K_{s-1}(\alpha_{s}-1)& 0& 0 & \cdots & 0 
\end{vmatrix}\nonumber\displaybreak[0]\\
&=\pm q^{s}|C|\begin{vmatrix}
 K_{0}(\alpha_1) &\cdots & K_{s-1}(\alpha_1)&  K_{0}(\alpha_1) &\cdots &  K_{s-1}(\alpha_1) \cr
  \vdots &\ddots & \vdots & \vdots & \vdots \cr
   K_{0}(\alpha_s) &\cdots & K_{s-1}(\alpha_s) & K_{0}(\alpha_s) &\cdots & K_{s-1}(\alpha_s) \cr
 K_{0}(\alpha_{1}-1) &\cdots & K_{s-1}(\alpha_{1}-1)& 0 & \cdots & 0 \cr
 \vdots &\ddots & \vdots & \vdots & \ddots & \vdots \cr
 K_{0}(\alpha_{s}-1) &\cdots & K_{s-1}(\alpha_{s}-1)&  0 & \cdots & 0 
\end{vmatrix}\nonumber\\
&=\pm q^{s}|C|\begin{vmatrix}
   K_{0}(\alpha_1) &\cdots &  K_{s-1}(\alpha_1) \cr
   \vdots & \ddots & \vdots \cr
    K_{0}(\alpha_s) &\cdots & K_{s-1}(\alpha_s)  
\end{vmatrix}\begin{vmatrix}
 K_{0}(\alpha_1-1) &\cdots &  K_{s-1}(\alpha_1-1) \cr
   \vdots & \ddots & \vdots \cr
    K_{0}(\alpha_s-1) &\cdots & K_{s-1}(\alpha_s-1)   
\end{vmatrix}\nonumber\\ %label{eq:2}
&\text{(by Lemma~\ref{lem:det})}\nonumber\\
&=\pm q^{s}|C| 
\left(\frac{q^{s(s-1)/2}\prod_{1\leq i<j \leq s}(\alpha_i-\alpha_j)}{\prod_{i=1}^{s-1}i!}\right)^2.\nonumber
\end{align}

The fact that $(\det Q)/|C|$ is an integer follows from %Eq. \eqref{eq:2}. 
the second equality in the above equation.
\end{proof}
\begin{theorem}\label{thm:det}
If $C$ is a nontight $(2s-1)$-design with degree $s$ and $S(C)=\{\alpha_1,\ldots,\alpha_s\}$ in $H(n,q)$,  
then $$q^{s} 
\left(\frac{q^{s(s-1)/2}\prod_{1\leq i<j \leq s}(\alpha_i-\alpha_j)}{\prod_{i=1}^{s-1}i!}\right)^2$$ is an integer dividing $|C|^{2s}$.
\end{theorem}
\begin{proof}
Follows from Lemma~\ref{lem:Qint}, Lemma~\ref{lem:s}, and Theorem~\ref{thm:qant4}.
\end{proof}
\begin{remark}
Let $q=p^m$ where $p$ is a prime. Let $C$ be a linear code of length  $n$ with dimension $k$ over the finite field $\mathbb{F}_q$ with $q$ elements. 
Assume that $C$  is a $5$-design with degree $3$ in $H(n,q)$. 
By \cite[Corollary~3]{CG}, the three distances of $C$ are
$\alpha_1=(a-1)p^t,\alpha_2=ap^t,\alpha_3=(a+1)p^t$ for some integers $a,t$. 
Then Theorem~\ref{thm:CG} shows that $k\geq m+t$, while 
Theorem~\ref{thm:det} shows $k\geq \frac{3}{2}m+t$.
%\alg{[8,Corollary 2], cases (2) and (3)?}
\end{remark}

\begin{remark}
A theorem for tight designs, similar to Theorem~\ref{thm:det}, is stated as follows:  
If $C$ is a tight $(2s-1)$-design with degree $s$ and $S(C)=\{\alpha_1,\ldots,\alpha_s=n\}$, then 
$$q^{s} 
\left(\frac{q^{(s-1)(s-2)/2} \prod_{1\leq i<j \leq s-1}(\alpha_i-\alpha_j)}{\prod_{i=1}^{s-2}i!}\right)^2$$ is an integer dividing $|C|^{2s-1}$. 
%Note that the Hamming distances $\alpha_i$ are uniquely determined as the zeros of $\sum_{j=0}^{s-1} K_{n-1,q,j}(x)$.
%\textcolor{red}{I dont understand the next sentence.}
%Then by Theorem~\ref{thm:wilson}(1), $\alpha_1,\ldots,\alpha_{s-1}$ are the zeros of the polynomial $\sum_{j=0}^{s-1}K_{n-1,q,j}(x)$. 
%We wonder if this condition is stronger than our result above. 
%\alg{comment on this later. Now I can't remember what did I mean by "our result above"?}
\end{remark}

\end{document}